\documentclass[11pt]{article}

\usepackage{amsmath} 
\usepackage{amsthm} 
\usepackage{amsfonts} 
\usepackage{amssymb} 
\usepackage{stmaryrd} 
\usepackage{mathtools} 
\usepackage{thm-restate} 

\usepackage{xifthen} 

\usepackage{xcolor} 
\usepackage{bm} 
\usepackage{bbm} 

\usepackage{array} 
\usepackage[inline]{enumitem} 

\usepackage{microtype} 
\usepackage{lmodern} 

\usepackage{pdflscape} 
\usepackage{afterpage} 

\usepackage{tikz} 
\usetikzlibrary{cd} 
\usetikzlibrary{calc} 

\usepackage[in]{fullpage} 

\usepackage[hypcap=true]{subcaption} 

\usepackage[%
  pdftex,
  plainpages=false,
  pdfpagelabels,
  colorlinks,
  citecolor=black,
  linkcolor=black,
  urlcolor=black,
  filecolor=black,
  bookmarksopen=false
]{hyperref} 
\usepackage[all]{hypcap} 

\newtheoremstyle{thmstyle}
  {\medskipamount}
  {\smallskipamount}
  {\slshape}
  {0pt}
  {\bfseries}
  {.}
  { }
  {\thmname{#1}\thmnumber{ #2}{\normalfont\thmnote{ (#3)}}}
  
\newtheoremstyle{plainstyle}
  {\medskipamount}
  {\smallskipamount}
  {\rmfamily}
  {0pt}
  {\bfseries}
  {.}
  { }
  {\thmname{#1}\thmnumber{ #2}{\normalfont\thmnote{ (#3)}}}

\theoremstyle{thmstyle}
\newtheorem{theorem}{Theorem}[section]
\newtheorem{lemma}[theorem]{Lemma}
\newtheorem{corollary}[theorem]{Corollary}
\newtheorem{proposition}[theorem]{Proposition}

\theoremstyle{plainstyle}
\newtheorem{definition}[theorem]{Definition}

\newtheorem{conjecture}[theorem]{Conjecture}

\newlist{enumdef}{enumerate}{1}
\setlist[enumdef]{before={\leavevmode}, label={\arabic*.}, ref={\thetheorem.\arabic*}}

\setlist[enumerate]{label={\roman*.}, ref={(\roman*)}} 

\numberwithin{equation}{section} 

\let\epsilon\varepsilon

\newcommand{\df}{\stackrel{\text{def}}{=}}

\newcommand{\comp}{\mathbin{\circ}}
\newcommand{\rest}{\mathord{\vert}}

\newcommand{\up}{\mathord{\uparrow}}

\DeclareMathOperator{\im}{im}
\DeclareMathOperator{\id}{id}
\DeclareMathOperator{\Hom}{Hom}
\DeclareMathOperator{\Aut}{Aut}
\DeclareMathOperator{\Fix}{Fix}
\DeclareMathOperator{\Fold}{Fold}
\DeclareMathOperator{\IndFold}{IndFold}
\DeclareMathOperator{\FoldGroup}{FoldGroup}
\DeclareMathOperator{\Spec}{Spec}

\newcommand{\NN}{\mathbb{N}}

\newcommand{\RR}{\mathbb{R}}

\newcommand{\One}{\mathbbm{1}}

\newcommand{\cC}{\mathcal{C}}
\newcommand{\cD}{\mathcal{D}}
\newcommand{\cF}{\mathcal{F}}
\newcommand{\cM}{\mathcal{M}}
\newcommand{\cR}{\mathcal{R}}

\def\Holder{H\"{o}lder}
\def\Lovasz{Lov\'{a}sz}

\title{Bigraph percolation problems}
\author{%
  Leonardo N.~Coregliano%
}
\date{\today}

\begin{document}
\maketitle

\begin{abstract}
  A bigraph $G$ is weakly norming if the $e(G)$th root of the density of $G$ in $\lvert W\rvert$ is
  a norm in the space of bounded measurable functions $W\colon\Omega\times\Lambda\to\RR$. The only
  known technique, due to Conlon--Lee, to show that a bigraph $G$ is weakly norming is to present a
  cut-percolation sequence of $G$.

  In this paper, we identify a key obstacle for cut-percolation, which we call fold-stability and we
  show that existence of a cut-percolating of a bigraph $G$ is equivalent to non-existence of
  non-monochromatic fold-stable colorings of the edges of $G$.
\end{abstract}

\section{Introduction}

Given a finite bigraph $G$ (i.e., a bipartite graph with a fixed bipartition $(V_1(G),V_2(G))$) and
a bounded measurable function $W\colon\Omega\times\Lambda\to\RR$, where $\Omega=(X,\mu)$ and
$\Lambda=(Y,\nu)$ are probability spaces, the \emph{homomorphism density} of $G$ in $W$, defined as
\begin{equation*}
  t(G,W) \df \int_{X^{V_1(G)}\times Y^{V_2(G)}} \prod_{(u,v)\in E(G)} W(x_u,y_v)\ d(\mu\otimes\nu)(x,y),
\end{equation*}
forms the backbone of the theory of graphons.

Two of the early questions of the field~\cite[\S~14.1]{Lov12} were to characterize which bigraphs
$G$ are \emph{norming} (\emph{weakly norming}, respectively) in the sense that $t(G,W)^{1/e(G)}$
($t(G,\lvert W\rvert)^{1/e(G)}$, respectively) is a norm in the vector space of bounded measurable
functions $\Omega\times\Lambda\to\RR$ (up to zero-measure change).

The first major step on this characterization was taken by Hatami by relating the problem to a
specific \Holder\ property. Namely, every bigraph $G$ naturally defines a homogeneous function
\begin{equation*}
  t(G;(W_e)_{e\in E(G)})
  \df
  \int_{X^{V_1(G)}\times Y^{V_2(G)}} \prod_{(u,v)\in E(G)} W_{(u,v)}(x_u,y_v)\ d(\mu\otimes\nu)(x,y)
\end{equation*}
that takes in a sequence $(W_e)_{e\in E(G)}$ of bounded measurable functions
$\Omega\times\Lambda\to\RR$. Hatami showed that $G$ is norming (weakly norming, respectively) if and
only if for all sequences $(W_e)_{e\in E(G)}$ of bounded measurable functions
$\Omega\times\Lambda\to\RR$ ($\Omega\times\Lambda\to\RR_+$, respectively), we have
\begin{equation}\label{eq:Holder}
  t(G;(W_e)_{e\in E(G)}) \leq \prod_{e\in E(G)} t(G,W_e)^{1/e(G)}.
\end{equation}

Some examples of (weakly) norming bigraphs are the following:
\begin{itemize}
\item Even cycles $C_{2n}$ are norming. The corresponding norms are the Schatten--von-Neumann norms:
  \begin{equation*}
    t(C_{2n},W)^{1/(2n)} = \left(\sum_{\lambda\in\Spec(W)} \lvert\lambda\rvert^{2n}\right)^{1/(2n)},
  \end{equation*}
  where $\Spec(W)$ is the spectrum of the natural compact operator $L_2(\Omega\times\Lambda)\to
  L_2(\Omega\times\Lambda)$ associated to $W$.
\item Complete bigraphs $K_{p,q}$ are weakly norming and norming when both $p$ and $q$ are
  even~\cite[Theorem~2.9(ii)]{Hat10}. When $p=1$, we get the $L_q$-norm over $\Omega$ of the
  $L_1$-norm over $\Lambda$:
  \begin{equation*}
    t(K_{1,q},\lvert W\rvert)^{1/q}
    =
    \left(\int_X\left(\int_Y \lvert W(x,y)\rvert\ d\nu(y)\right)^q\ d\mu(x)\right)^{1/q}.
  \end{equation*}
  When $q=1$, we get the $L_q$-norm over $\Lambda$ of the $L_1$-norm over $\Omega$.

  More generally, when $p,q\geq 2$, we get the ``$L_q$-norm of the common neighborhood of $p$
  vertices in $\Omega$'' (which is the same as the ``$L_p$-norm of the common neighborhood of $q$
  vertices in $\Lambda$'')
  \begin{align*}
    t(K_{p,q},\lvert W\rvert)^{1/(pq)}
    & =
    \left(\int_{X^p}\left(\int_Y \prod_{i=1}^p\lvert W(x_i,y)\rvert\ d\nu(y)\right)^q\ d\mu(x)\right)^{1/(pq)}
    \\
    & =
    \left(\int_{Y^q}\left(\int_X \prod_{i=1}^q\lvert W(x,y_i)\rvert\ d\mu(x)\right)^p\ d\nu(y)\right)^{1/(pq)}.
  \end{align*}
\item A balanced complete bigraph minus a perfect matching $K_{n,n} - n K_{1,1}$ is a weakly norming
  bigraph~\cite[Proposition~14.2]{Lov12}.
\item Hypercubes $Q_n$ are weakly norming~\cite[Theorem~2.9(iii)]{Hat10}.
\end{itemize}

For the particular case of hypercubes $Q_n$, Hatami used a technique that was later generalized and
systematically studied by Conlon--Lee~\cite{CL17} under the name cut-percolation. For this
technique, it is convenient to restate~\eqref{eq:Holder} using a version of homomorphism densities
for colored bigraphs. Namely, let us color the edges of our bigraph $G$ with $c\colon E(G)\to C$ and
for a sequence $W=(W_i)_{i\in C}$ of bounded measurable functions $\Omega\times\Lambda\to\RR$,
define
\begin{equation}\label{eq:tGcW}
  t((G,c),W)
  \df
  \int
  \int_{X^{V_1(G)}\times Y^{V_2(G)}} \prod_{(u,v)\in E(G)} W_{c(u,v)}(x_u,y_v)\ d(\mu\otimes\nu)(x,y).
\end{equation}
Then~\eqref{eq:Holder} can be easily restated as
\begin{equation}\label{eq:Holdercolor}
  t((G,c),W) \leq \prod_{i\in\im(c)} t((G,i),W)^{1/e(G)},
\end{equation}
for every \emph{injective} coloring $c\colon E(G)\to C$ and every sequence $W=(W_i)_{i\in C}$ of bounded measurable
functions $\Omega\times\Lambda\to\RR$, where in the right-hand side, $i$ denotes the constant
coloring that colors every edge of $G$ with the color $i$.

For the remainder of this introduction, we assume $G$ is connected and on the weak norming property,
so we can assume without loss of generality that all our $W_i$ are non-negative.

The first idea to show~\eqref{eq:Holdercolor} is to apply Cauchy--Schwarz Inequality through
``folds'' of the bigraph $G$. Namely, a fold of $G$ is a pair $(f,L)$ such that
\begin{itemize}
\item $f\in\Aut(G)$ is an automorphism of $G$ that is an involution (i.e., $f^{-1} = f$).
\item The set of fixed points $\Fix(f)$ is a vertex-cut of $G$.
\item $L$ is a union of connected components of $G-\Fix(f)$.
\item $(L,\Fix(f),f(L))$ is a partition of $V(G)$.
\end{itemize}
Whenever $(f,L)$ is a fold of $G$, we can nicely split the integral in~\eqref{eq:tGcW} as follows:
\begin{equation*}
  t((G,c),W)
  =
  \int_{X^{V_1(G)\cap\Fix(f)}\times Y^{V_2(G)\cap\Fix(f)}}
  F(x,y)\cdot F_L(x,y)\cdot F_{f(L)}(x,y)
  \ d(\mu\otimes\nu)(x,y)
\end{equation*}
where $F(x,y)$ collects all terms corresponding to edges completely contained in $\Fix(f)$,
$F_L(x,y)$ collects all terms corresponding to edges that have at least one endpoint in $L$ and
$F_{f(L)}(x,y)$ collects all terms corresponding to edges that have at least one endpoint in $f(L)$
(since $\Fix(f)$ is a vertex-cut, there are no edges between $L$ and $f(L)$). By applying
Cauchy--Schwarz Inequality and recalling that $f$ is an automorphism of $G$, we get
\begin{equation}\label{eq:foldCS}
  \begin{aligned}
    t((G,c),W)
    & \leq
    \begin{multlined}[t]
      \left(
      \int_{X^{V_1(G)\cap\Fix(f)}\times Y^{V_2(G)\cap\Fix(f)}}
      F(x,y)\cdot F_L(x,y)^2
      \ d(\mu\otimes\nu)(x,y)\right)^{1/2}
      \\
      \cdot
      \left(
      \int_{X^{V_1(G)\cap\Fix(f)}\times Y^{V_2(G)\cap\Fix(f)}}
      F(x,y)\cdot F_{f(L)}(x,y)^2
      \ d(\mu\otimes\nu)(x,y)\right)^{1/2}
    \end{multlined}
    \\
    & =
    \sqrt{t((G,c\comp f_L),W)\cdot t((G,c\comp f_L^*),W)},
  \end{aligned}
\end{equation}
where $f_L,f_L^*\colon V(G)\to V(G)$ are the \emph{folding maps} given by
\begin{align*}
  f_L(v) & \df
  \begin{dcases*}
    f(v), & if $v\in\Fix(f)\cup f(L)$,\\
    v, & if $v\in L$,
  \end{dcases*}
  &
  f_L^*(v) & \df
  \begin{dcases*}
    f(v), & if $v\in L$,\\
    v, & if $v\in\Fix(f)\cup f(L)$.
  \end{dcases*}
\end{align*}

By repeatedly applying the idea above using different folds, one obtains a chain of inequalities
that is nicely represented by a finite binary tree $\psi$ such that:
\begin{enumerate}
\item\label{it:root} The root of $\psi$ is labeled by the coloring $c$.
\item\label{it:children} For each internal $\eta$ node of $\psi$, there is a fold $(f,L)$ of $G$
  such that if $\eta$ is labeled by a coloring $c'$ of $G$, then its left and right children are
  labeled by $c\comp f_L$ and $c\comp f_L^*$, respectively.
\end{enumerate}

Each such tree $\psi$ yields the inequality
\begin{equation*}
  t((G,c),W) \leq \prod_{\eta\in L(\psi)} t((G,\psi(\eta)),W)^{2^{-\lvert\eta\rvert}},
\end{equation*}
where $L(\psi)$ is the set of leaves of $\psi$, $\psi(\eta)$ is the coloring of $G$ that labels the
node $\eta$ and $\lvert\eta\rvert$ is the height of $\eta$ in the tree $\psi$.

A priori, the goal is to find a tree as above in which all leaves are labeled by monochromatic
colorings. This is not always possible (for example, it is impossible for $C_6$), so instead one
finds a sequence $(\psi_n)_{n\in\NN}$ of trees such that
\begin{equation*}
  M(\psi_n)
  \df
  \sum_{\substack{\eta\in L(\psi_n)\\\psi_n(\eta)\text{ monochromatic}}} 2^{-\lvert\eta\rvert}
\end{equation*}
converges to $1$. Using the Dominated Convergence Theorem and a symmetrization argument, one can
then conclude that~\eqref{eq:Holdercolor} holds.

A key observation of Conlon--Lee~\cite[Theorem~3.2]{CL17} is that the following are equivalent:
\begin{enumerate}[label={\Alph*.}, ref={(\Alph*)}]
\item\label{it:reach} For every coloring $c\colon E(G)\to C$, there exists a coloring tree $\psi$
  satisfying items~\ref{it:root} and~\ref{it:children} such that one of the leaves of $\psi$ is a
  monochromatic coloring.
\item\label{it:lim1} For every coloring $c\colon E(G)\to C$, there exists a sequence of coloring
  trees $(\psi_n)_{n\in\NN}$ with $\psi_{n+1}$ extending $\psi_n$ and satisfying items~\ref{it:root}
  and~\ref{it:children} such that $\lim_{n\to\infty} M(\psi_n)=1$ .
\item\label{it:cutperc} There exists a sequence $E_0,E_1,\ldots,E_n\subseteq E(G)$ such that:
  \begin{itemize}
  \item We have $\lvert E_0\rvert=1$.
  \item For every $i\in[n]$, there exists a fold $(f,L)$ of $G$ such that $E_i=f_L^{-1}(E_{i-1})$.
  \item We have $E_n=E(G)$.
  \end{itemize}
\end{enumerate}
A bigraph $G$ satisfying any (equivalently, all) of the above is called \emph{cut-percolating}.

From item~\ref{it:cutperc} above, it is clear that cut-percolating bigraphs are edge-transitive
under the action of the subgroup $\FoldGroup(G)$ of $\Aut(G)$ generated by $\{f\mid (f,L)\text{ is a
  fold of } G\}$. Conlon--Lee then conjectured that the converse also holds:
\begin{conjecture}[\protect{\cite[Conjecture~6.1]{CL17}}]\label{conj:cutperc}
  If $G$ is a bigraph that is edge-transitive under the action of $\FoldGroup(G)$, then $G$ is
  cut-percolating.
\end{conjecture}

In fact, in~\cite{CL17}, Conlon--Lee proved that reflection bigraphs, which are particular bigraphs
obtained via reflection groups, are cut-percolating and stated the stronger conjecture that all of
the following classes are in fact the same:
\begin{enumerate*}
\item reflection bigraphs,
\item cut-percolating bigraphs,
\item weakly norming bigraphs,
\item bigraphs that are edge-transitive under the action of $\FoldGroup(G)$.
\end{enumerate*}

Toward proving this stronger conjecture, a few other characterizations of weakly norming bigraphs
were obtained in~\cite{KMPW19,LS21,LS22}. Also, significant progress was obtained recently by
Sidorenko~\cite{Sid20}, showing that weakly norming bigraphs are necessarily edge-transitive.

Let us now return to Conjecture~\ref{conj:cutperc} and let us observe that there are three kinds of
obstacles for a bigraph $G$ to be cut-percolating.

First, let us call a coloring $c\colon E(G)\to C$ \emph{fold-stable} if for every fold $(f,L)$ of
$G$, the colored bigraphs $(G,c)$ and $(G,c\comp f_L)$ are isomorphic. A simple induction shows that
if we attempt to construct a coloring tree $\psi$ as in item~\ref{it:reach} starting from $c$ at the
root, then every node $\eta$ of $\psi$ will be labeled by some coloring $\psi(\eta)$ such that
$(G,c)\cong (G,\psi(\eta))$, which in particular means that $\psi(\eta)$ is not monochromatic. Thus,
the first kind of obstacle for $G$ to be cut-percolating is the existence of a non-monochromatic
fold-stable coloring.

Second, let us call a coloring $c\colon E(G)\to C$ \emph{strongly fold-stable} if for every fold
$(f,L)$, there exists an automorphism $h\in\Aut((G,c))$ of the colored bigraph $(G,c)$ such that
$\Fix(h)=\Fix(f)$ and $(h,L)$ is a fold of $G$. Since $h$ also preserves $c$, it is straightforward
to check that $c = c\comp f_L\comp h_L = c\comp f_L^*\comp h_L^*$, that is, whatever a folding map
can do, there is a folding map that undoes it. As the name suggests, we will show later that every
strongly fold-stable coloring is necessarily fold-stable. A priori, existence of a non-monochromatic
strongly fold-stable coloring $c$ is an even worse obstacle than the first as not only all colorings in
a coloring tree starting from $c$ will be isomorphic to $c$, but all foldings that occur on the tree
can be inverted in a single step.

Finally, the last obstacle is more abstract. Let us write $c\preceq_{\Fold(G)} c'$ if we can
construct a coloring tree starting from $c$ and containing $c'$ in one of its nodes. It is easy to
see that $\preceq_{\Fold(G)}$ is a (partial) preorder. Similarly to the situation of fold-stable
colorings, if $c$ is a non-monochromatic $\preceq_{\Fold(G)}$-maximal coloring (i.e., every $c'$
satisfying $c\preceq_{\Fold(G)} c'$ also satisfies $c'\preceq_{\Fold(G)} c$), then in any coloring
tree $\psi$ starting at $c$, all colorings $c'$ that appear in $\psi$ must satisfy
$c'\preceq_{\Fold(G)} c$, hence must also not be monochromatic.

In this paper we show that all these obstacles are in fact the same and their absence completely
characterizes cut-percolation:
\begin{theorem}[(A particular case of) Theorem~\ref{thm:cutperc}]
  The following are equivalent for a finite connected bigraph $G$.
  \begin{enumerate}
  \item $G$ is cut-percolating.
  \item Every fold-stable coloring of $G$ is monochromatic.
  \item Every strongly fold-stable coloring of $G$ is monochromatic.
  \item Every $\preceq_{\Fold(G)}$-maximal coloring of $G$ is monochromatic.
  \item $G$ is edge-transitive under $\FoldGroup(G)$ and every symmetrically fold-stable coloring of
    $G$ is monochromatic.
  \end{enumerate}
\end{theorem}
In the above, a \emph{symmetrically fold-stable} coloring $c\colon E(G)\to C$ is a coloring that is
strongly fold-stable and for every $i,j\in\im(c)$, there exists a permutation $\sigma\colon C\to C$
such that $\sigma(i)=j$ and $(G,c)\cong(G,\sigma\comp c)$ (in particular, all color classes in $c$
yield isomorphic bigraphs).

In particular, a consequence of the theorem above is that Conjecture~\ref{conj:cutperc} can be
rephrased as follows:
\begin{conjecture}
  If $G$ is a bigraph that is edge-transitive under the action of $\FoldGroup(G)$, then every
  symmetrically fold-stable coloring of $G$ is monochromatic.
\end{conjecture}
The hope would be to use the strong symmetry properties of symmetrically fold-stable colorings to
attack the conjecture above (or to disprove it by presenting a bigraph $G$ that is edge-transitive
under the action of $\FoldGroup(G)$ along with a non-monochromatic fold-stable coloring).

\subsection{More general cut-percolation}

Even though the main focus during the development of the current work was
Conjecture~\ref{conj:cutperc}, its arguments extend nicely to other related problems involving folds
of bigraphs. Three examples that are worth mentioning are the following:
\begin{enumerate}[label={\arabic*.}, wide]
\item When studying norming bigraphs (as opposed to weakly norming ones), the same arguments of
  cut-percolation hold using instead folds whose corresponding cut-involution's fixed set is an
  independent set of the bigraph~\cite{CL17}.
\item In~\cite{CL24}, the more general problem of \emph{domination between bigraphs} is
  studied. Namely, a bigraph $G_1$ dominates $G_2$ if $t(G_1,W)^{1/e(G_1)}\geq t(G_2,W)^{1/e(G_2)}$
  for every bounded measurable function $W\colon\Omega\times\Lambda\to\RR_+$ (e.g., Sidorenko's
  Conjecture amounts every bigraph dominating the edge bigraph). Two main operations are used to
  obtain domination between bigraphs in~\cite{CL24}: fold inequalities as in~\eqref{eq:foldCS} and a
  relocation operation (which we will not cover here).
\item In~\cite{Cor24}, with the aim of studying Sidorenko's Conjecture, the author observed that a
  left-sided analogue of cut-percolation can be used to obtain induced-Sidorenko bigraphs, which
  together with the results of~\cite{CR21} can obtain several more Sidorenko bigraphs than one
  usually obtains from the analogous techniques that use cut-percolation instead. In
  left-cut-percolation, we start with a colored bigraph $H=(G,c)$, restrict ourselves to folds of
  $G$ whose involution preserves the coloring $c$ and in analogy to item~\ref{it:cutperc}, we would
  like to find a sequence $V_0,V_1,\ldots,V_n\subseteq V_1(G)$ such that:
  \begin{itemize}
  \item We have $\lvert V_0\rvert=1$.
  \item For every $i\in[n]$, there exists a fold $(f,L)\in\cF$ such that $V_i=f_L^{-1}(V_{i-1})$.
  \item We have $V_n=V_1(G)$.
  \end{itemize}

  A similar theory is then developed to show that the above are equivalent to items analogous
  to~\ref{it:reach} and~\ref{it:lim1}.
\end{enumerate}

To study problems such as the above, we can extend the mentioned notions as follows:
\begin{enumerate}[label={\arabic*.}]
\item We define the refined notion of $(K,\cF)$-fold-stability for a subgroup $K$ of $\Aut(G)$ and a
  subset $\cF$ of folds by restricting the original notion to apply only to folds in $\cF$ and require the
  automorphism between $(G,c)$ and $(G,c\comp f_L)$ to be in $K$.
\item Similarly, we define the refined notion of $\cF'$-strongly $(K,\cF)$-fold-stability by
  requiring $(K,\cF)$-fold-stability and requiring the ``inverse fold'' $(h,L)$ to be in $\cF'$.
\item For a set $\cF$ of folds of $G$, a set $\cM$ of ``objective'' colorings of $G$ and a coloring
  $c$, we write $c\preceq_\cF\cM$ if there exists a coloring tree satisfying items~\ref{it:root}
  and~\ref{it:children} using only folds in $\cF$ and containing some coloring in $\cM$ as labeling
  one of its nodes, that is, we can ``reach'' $\cM$ from $c$ using folds in $\cF$.

  We write $c\ll_\cF\cM$ if we can find a sequence of trees $\psi_n$ satisfying items~\ref{it:root}
  and~\ref{it:children} using only folds in $\cF$ such that for
  \begin{equation*}
    M(\psi_n)
    \df
    \sum_{\substack{\eta\in L(\psi_n)\\\psi_n(\eta)\in\cM}} 2^{-\lvert\eta\rvert}
  \end{equation*}
  we have $\lim_{n\to\infty} M(\psi_n) = 1$. That is, starting from $c$, we can ``percolate coloring
  trees''\footnote{This is an a priori different notion of ``percolation''.} to $\cM$
\end{enumerate}

In this paper, we cover sufficient conditions for equivalence of $c\preceq_\cF\cM$ and $c\ll_\cF\cM$
(Proposition~\ref{prop:abs}) and conditions for equivalence between the notions of fold-stability,
strong fold-stability and maximality (Corollary~\ref{cor:obst}). This allows us to characterize
versions of cut-percolation (and left-cut-percolation) in which only a subset of folds can be used
in terms of the corresponding aforementioned three obstacles (Theorems~\ref{thm:cutperc}
and~\ref{thm:leftcutperc}).

\subsection{The key ideas of the argument}

Recall that the motivation to study cut-percolation is that it reveals inequalities between
different $t((G,c),W)$. The key idea behind Theorem~\ref{thm:cutperc} is to leverage these
inequalities to understand combinatorial properties of colorings.

As an example, let us give a high-level idea of how we show that if $G$ is not cut-percolating, then
there must exist a strongly fold-stable coloring that is not monochromatic. From the description of
item~\ref{it:reach}, there exists a coloring $c$ that cannot reach any monochromatic coloring via a
coloring tree (i.e., $c\not\preceq_{\Fold(G)}\cM$ for the set $\cM$ of monochromatic colorings of
$G$).

Since the bigraph is finite, the set of colorings reachable from $c$ is finite, so it must contain
some $\preceq_{\Fold(G)}$-maximal coloring $c'$. In turn, maximality of $c'$ implies that in every
coloring tree we construct starting from $c'$, all nodes can reach back to $c'$; in particular, for
every fold $(f,L)$, not only must have $c'\comp f_L\preceq_{\Fold(G)} c'$ and $c'\comp
f_L^*\preceq_{\Fold(G)} c'$, but in fact, we must have $c'\comp f_L\ll_{\Fold(G)} \{c'\}$ and
$c'\comp f_L^*\ll_{\Fold(G)}\{c'\}$. If we recall the inequalities corresponding to these
coloring trees, we get
\begin{equation*}
  t((G,c'),W)
  \leq
  \sqrt{t((G,c'\comp f_L),W)\cdot t((G,c'\comp f_L^*),W)}
  \leq
  t((G,c'),W),
\end{equation*}
so we must have equality throughout. Thus, since the first inequality in the above corresponds to an
application of Cauchy--Schwarz, equality in the above implies some kind of linear dependence between
certain functionals. This linear dependence will allow us to show (Proposition~\ref{prop:coreiso})
isomorphism between certain partially labeled colored bigraphs (i.e., \emph{flags} in the language
of flag algebras~\cite{Raz07}), which will eventually lead to (strong) fold-stability of $c'$ (which
is necessarily \emph{not} monochromatic, as it is reachable from $c$).

\subsection{Organization}

Section~\ref{sec:defs} contains all definitions and notation used in the
paper. Section~\ref{sec:basicprops} contains basic properties about folds that follow almost
immediately from definitions (but will be needed for later arguments). Section~\ref{sec:flagiso} is
devoted to showing Proposition~\ref{prop:coreiso}, which provides an isomorphism between flags when
the corresponding density functionals are linearly dependent. Section~\ref{sec:reachvsperc} is
devoted to showing a sufficient condition (\emph{absorption}) for the equivalence of
$c\preceq_\cF\cM$ and $c\ll_\cF\cM$ (Proposition~\ref{prop:abs}). Section~\ref{sec:ineqsperc}
collects all basic relations between folds of bigraphs and inequalities that they induce on bounded
measurable functions. In Section~\ref{sec:foldstability}, we establish all implications between the
notions of fold-stability, strong fold-stability and $\preceq_\cF$-maximality needed for the final
theorems, which are proved in Section~\ref{sec:cutperc}. Finally, the flag isomorphism obtained in
Proposition~\ref{prop:coreiso} is only between what we call ``connected cores'' of the flags;
Appendix~\ref{sec:coreiso} shows how to upgrade this to a full flag isomorphism if the underlying
bigraphs are isomorphic (this is not used anywhere in the argument, but nicely complements
Proposition~\ref{prop:coreiso}).

\section{Basic definitions and notation}
\label{sec:defs}

In this section, we establish notation and provide all definitions used in the paper. These are
grouped thematically and are ordered in descending order of (probable) familiarity to a
combinatorics reader.

We denote the set non-negative integers by $\NN\df\{0,1,\ldots\}$, the set of positive integers by
$\NN_+$ and the set of non-negative reals by $\RR_+$. For $n\in\NN$, we let
$[n]\df\{1,2,\ldots,n\}$.

\begin{definition}[Bigraphs]
  \begin{enumdef}
  \item A \emph{bigraph} is a tuple $G=(V_1,V_2,E)$ where $V_1$ and $V_2$ are disjoint sets and
    $E\subseteq V_1\times V_2$. We use the shorthand notations:
    \begin{align*}
      V_1(G) & \df V_1, & V_2(G) & \df V_2, & E(G)\df E,\\
      v_1(G) & \df \lvert V_1\rvert, & v_2(G) & \df \lvert V_2\rvert, & e(G)\df \lvert E\rvert.
    \end{align*}
    We also let $V(G)\df V_1(G)\cup V_2(G)$ and $v(G)\df\lvert V(G)\rvert$.

    We say that $G$ is \emph{finite} if $V(G)$ is finite.

  \item A \emph{bigraph homomorphism} from a bigraph $G_1$ to a bigraph $G_2$ is a function $f\colon
    V(G_1)\to V(G_2)$ such that
    \begin{align*}
      f(V_1(G_1)) & \subseteq V_1(G_2), &
      f(V_2(G_1)) & \subseteq V_2(G_2), &
      f(E(G_1)) & \subseteq E(G_2),
    \end{align*}
    where in the last one, we make a small abuse of notation, by denoting by $f$ the product
    function $V_1(G_1)\times V_2(G_2)\to V_1(G_2)\times V_2(G_2)$ that acts as $f$ in both
    coordinates. The set of bigraph homomorphisms from $G_1$ to $G_2$ is denoted $\Hom(G_1,G_2)$.

  \item A \emph{bigraph embedding} is a bigraph homomorphism $f$ that is injective and further
    satisfies
    \begin{equation*}
      f((V_1(G_1)\times V_2(G_1))\setminus E(G_1)) \subseteq (V_1(G_2)\times V_2(G_2))\setminus E(G_2).
    \end{equation*}

    A \emph{bigraph isomorphism} is a bijective bigraph embedding. Two bigraphs $G_1$ and $G_2$
    are \emph{isomorphic} (denoted $G_1\cong G_2$) if there exists a bigraph isomorphism from $G_1$
    to $G_2$.

    An \emph{automorphism} of a bigraph $G$ is an isomorphism from $G$ to $G$. We denote the
    \emph{group of automorphisms} of $G$ by $\Aut(G)$.

  \item If $G$ is a bigraph and $K$ is a subgroup of $\Aut(G)$, we say that $G$ is
    \begin{description}
    \item[$K$-edge-transitive] if the induced action of $K$ on $E(G)$ is
      transitive, i.e., for every $e_1,e_2\in E(G)$, there exists $h\in K$ with $h(e_1)=e_2$.
    \item[$K$-left-vertex-transitive] if the induced action of $K$ on $V_1(G)$
      is transitive, i.e., for every $u,v\in V_1(G)$, there exists $h\in K$ with $h(v_1)=v_2$.
    \end{description}
    We omit $K$ from the notation when $K=\Aut(G)$.

  \item Let $G$ be a bigraph and $S\subseteq V(G)$.

    The \emph{bigraph induced by $S$ in $G$} is
    \begin{equation*}
      G\rest_S \df (V_1(G)\cap S, V_2(G)\cap S, E(G)\cap (S\times S)).
    \end{equation*}

    We also define $G - S\df G\rest_{V(G)\setminus S}$.

    For $E\subseteq E(G)$, we let
    \begin{equation*}
      G - E \df (V_1(G),V_2(G),E(G)\setminus E)
    \end{equation*}
    be the bigraph obtained from $G$ by removing the edges in $E$.

  \item Let $G$ be a bigraph and $v\in V(G)$. The \emph{neighborhood} of $v$ in $G$ is
    \begin{equation*}
      N_G(v) \df
      \begin{dcases*}
        \{u\in V_2(G) \mid (v,u)\in E(G)\}, & if $v\in V_1(G)$,\\
        \{u\in V_1(G) \mid (u,v)\in E(G)\}, & if $v\in V_2(G)$.
      \end{dcases*}
    \end{equation*}

  \item Let $G$ be a bigraph.

    A \emph{path} in $G$ from $u\in V(G)$ to $v\in V(G)$ is a finite sequence $w_0,\ldots,w_n$ in
    $V(G)$ such that $w_0=u$, $w_n=v$ and for every $i\in[n]$, we have $(w_{i-1},w_i)\in E(G)$ or
    $(w_i,w_{i-1})\in E(G)$.

    A set $S\subseteq V(G)$ is \emph{connected} in $G$ if for every $u,v\in S$, there exists a path
    in $G$ from $u$ to $v$.

    A \emph{connected component} of $G$ is a maximal connected non-empty set in $G$, that is, it is
    a non-empty subset of $V(G)$ that is connected in $G$ and is not properly contained in any other
    set that is connected in $G$.

    We say that $G$ is \emph{connected} if $V(G)$ is connected in $G$. We say that $G$ is
    \emph{disconnected} if $G$ is not connected.

    A \emph{(vertex-)cut} in $G$ is a set $S\subseteq V(G)$ such that $G-S$ is disconnected.

    A vertex $v\in V(G)$ is \emph{isolated} in $G$ if $\{v\}$ is a connected component of $G$.

    A set $S\subseteq V(G)$ is an \emph{independent set} of $G$ if it does not contain any edge of
    $G$, that is, if $E(G)\cap(S\times S) = \varnothing$.
  \end{enumdef}
\end{definition}

\begin{definition}[Colored bigraphs]
  \begin{enumdef}
  \item A \emph{coloring} of a bigraph $G$ is a function $c\colon E(G)\to C$, where $C$ is any
    set. We use the shorthand notation $C_c\df C$.

    A \emph{colored bigraph} is a pair $H=(G,c)$, where $G$ is a bigraph and $c$ is a coloring of
    $G$. We use the shorthand notations:
    \begin{align*}
      G(H) & \df G, & c_H & \df c, & C_H & \df C_c,\\
      V_1(H) & \df V_1(G), & V_2(H) & \df V_2(G), & E(H) & \df E(G),\\
      v_1(H) & \df v_1(G), & v_2(H) & \df v_2(G), & e(H) & \df e(G),\\
      V(H) & \df V(G), & v(H) & \df v(G).
    \end{align*}

    We say that $H$ is \emph{finite} if $G(H)$ is finite.

  \item We say that a coloring $c$ of a bigraph $G$ is:
    \begin{description}
    \item[monochromatic] if $\lvert\im(c)\rvert\leq 1$.
    \item[rainbow] if $c$ is injective.
    \end{description}
    By extension, we say that a colored bigraph $H$ is monochromatic (rainbow, resp.) if its
    coloring $c_H$ is so.

  \item A \emph{colored bigraph homomorphism} from a colored bigraph $H_1$ to a colored bigraph
    $H_2$ is a bigraph homomorphism $f$ from $G(H_1)$ to $G(H_2)$ that further \emph{preserves
    colors} in the sense that $c_{H_2}\comp f = c_{H_1}$ (note that this in particular forces
    $C_{H_1}=C_{H_2}$). The set of colored bigraph homomorphisms from $H_1$ to $H_2$ is denoted
    $\Hom(H_1,H_2)$.

  \item A \emph{colored bigraph embedding} is a bigraph embedding that is a colored bigraph
    homomorphism.

    A \emph{colored bigraph isomorphism} is a bijective colored bigraph embedding. Two colored
    bigraphs $H_1$ and $H_2$ are \emph{isomorphic} (denoted $H_1\cong H_2$) if there exists an
    isomorphism from $H_1$ to $H_2$.

    An \emph{automorphism} of a colored bigraph $H$ is an isomorphism from $H$ to $H$. We denote the
    group of automorphisms of $H$ by $\Aut(H)$.

  \item If $H$ is a colored bigraph and $K$ is a subgroup of $\Aut(H)$, we say that $H$ is
    \emph{$K$-edge-transitive} (\emph{$K$-left-vertex-transitive}, respectively) when $G(H)$ is
    so. We omit $K$ from the notation when $K=\Aut(H)$.

  \item Let $H$ be a colored bigraph.

    For $S\subseteq V(H)$, the \emph{colored bigraph induced by $S$ in $H$} is
    \begin{equation*}
      H\rest_S \df (G(H)\rest_S, c_H\rest_{E(G(H)\rest_S)})
    \end{equation*}
    We also define $H - S\df H\rest_{V(H)\setminus S}$.

    For $E\subseteq E(H)$, we let
    \begin{equation*}
      H - E \df (G(H)-E, c_H\rest_{E(G)\setminus E})
    \end{equation*}
    be the colored bigraph obtained from $H$ by removing the edges in $E$.

    For $D\subseteq C_H$, we let
    \begin{equation*}
      H_D \df ((V_1(H), V_2(H), c_H^{-1}(D)), c_H\rest_{c^{-1}(D)})
    \end{equation*}
    be the colored bigraph obtained from $H$ by keeping only the edges whose color is in $D$.

    For $i\in C_H$, we let $H_i\df G(H_{\{i\}})$ be the bigraph obtained from $H$ by keeping only
    the edges whose color is $i$ (and then forgetting the color).
  \end{enumdef}
\end{definition}

For the next definition, we will borrow a bit of the terminology from the theory of flag
algebras~\cite{Raz07}.

\begin{definition}[Flags]
  \begin{enumdef}
  \item A \emph{flag} is a pair $F = (H,\theta)$, where $H$ is a colored bigraph and
    $\theta\colon[k]\to V(H)$ is an injection for some $k\in\NN$. We use the shorthand notations:
    \begin{align*}
      H(F) & \df H, & \theta_F & \df \theta, & k_F & \df k,\\
      T(F) & \df \im(\theta), & T_1(F) & \df T(F)\cap V_1(H), & T_2(F) & \df T(F)\cap V_2(H),\\
      G(F) & \df G(H), & c_F & \df c_H, & C_F & \df C_H,\\
      V_1(F) & \df V_1(H), & V_2(F) & \df V_2(H), & E(F) & \df E(H),\\
      v_1(F) & \df v_1(H), & v_2(F) & \df v_2(H), & e(F) & \df e(H),\\
      V(F) & \df V(H), & v(F) & \df v(H).
    \end{align*}

    We say that $F$ is \emph{finite} if $H(F)$ is finite.

  \item A \emph{flag homomorphism} from a flag $F_1$ to a flag $F_2$ is a colored bigraph
    homomorphism $f$ from $H(F_1)$ to $H(F_2)$ that further \emph{preserves partial labeling} in the
    sense that $f\comp\theta_{F_1} = \theta_{F_2}$ (note that this in particular forces
    $k_{F_1}=k_{F_2}$). The set of flag homomorphisms from $F_1$ to $F_2$ is denoted
    $\Hom(F_1,F_2)$.

  \item A \emph{flag embedding} is a colored bigraph embedding that is a flag homomorphism.

    A \emph{flag isomorphism} is a bijective flag embedding. Two flags $F_1$ and $F_2$ are
    \emph{isomorphic} (denoted $F_1\cong F_2$) if there exists a flag isomorphism from $F_1$ to
    $F_2$.

    An \emph{automorphism} of flag $F$ is an isomorphism from $F$ to $F$. We denote the
    group of automorphisms of $F$ by $\Aut(F)$.

  \item We say that two flags $F_1$ and $F_2$ \emph{have the same type} if $k_{F_1}=k_{F_2}$ and
    $F_1\rest_{T(F_1)}\cong F_2\rest_{T(F_2)}$ (note that if this is the case, then the isomorphism
    must be given by $\theta_{F_2}\comp\theta_{F_1}^{-1}$ and we must have $C_{F_1}=C_{F_2}$).

  \item Let $F$ be a flag and $S\subseteq V(F)$.

    If $S\supseteq T(F)$, then the \emph{flag induced by $S$ in $F$} is
    \begin{equation*}
      F\rest_S \df (H(F)\rest_S, \theta')
    \end{equation*}
    where $\theta'\colon [k_F]\to S$ is the function obtained from $\theta$ by restricting its
    codomain to $S$.

    If $S\cap T(F)=\varnothing$, then we let $F-S\df F\rest_{V(F)\setminus S}$.

    For $E\subseteq E(F)$, we let
    \begin{equation*}
      F - E \df (H(F)-E, \theta)
    \end{equation*}
    be the flag obtained from $F$ by removing the edges in $E$.

  \item The \emph{connected core} $C(F)$ of a flag $F$ is the union of all connected components of
    $G(F)$ that contain at least one vertex of $T(F)$.

    We say that a flag $F$ is \emph{core-connected} if $C(F)=V(F)$.
  \end{enumdef}
\end{definition}

\begin{definition}[Bigraphons]
  \begin{enumdef}
  \item Given probability spaces $\Omega=(X,\mu)$ and $\Lambda=(Y,\nu)$, a \emph{bigraphon} over
    $\Omega$ and $\Lambda$ is a bounded measurable function $W\colon X\times Y\to\RR_+$, where
    $X\times Y$ is equipped with the product $\sigma$-algebra and the product measure
    $\mu\otimes\nu$; we will denote such a bigraphon by $W\colon\Omega\times\Lambda\to\RR_+$.

    If further $V$ is a finite set, we let $\Omega^V\df(X^V,\mu^V)$ be the product probability space
    equipped with the product measure $\mu^V$. By abuse of notation, we denote $\mu^V$ simply by
    $\mu$.

  \item Given a finite flag $F$ and a sequence $W = (W_i)_{i\in C_F}$ of bigraphons
    $W_i\colon\Omega\times\Lambda\to\RR_+$ (all over the same spaces), we define the
    \emph{homomorphism density function} $t(F,W)\colon X^{T_1(F)}\times Y^{T_2(F)}\to\RR_+$ by
    \begin{equation*}
      t(F,W)(x,y)
      \df
      \int_{X^{V_1(F)\setminus T_1(F)}\times Y^{V_2(F)\setminus T_2(F)}}
      \prod_{(u,v)\in E(F)} W_{c(u,v)}(x''_u,y''_v)
      \ d(\mu\otimes\nu)(x',y'),
    \end{equation*}
    where
    \begin{align*}
      x''_u & \df
      \begin{dcases*}
        x_u, & if $u\in T_1(F)$,\\
        x'_u, & if $u\in V_1(F)\setminus T_1(F)$,
      \end{dcases*}
      &
      y''_v & \df
      \begin{dcases*}
        y_v, & if $v\in T_2(F)$,\\
        y'_v, & if $v\in V_2(F)\setminus T_2(F)$.
      \end{dcases*}
    \end{align*}

    Given further a finite colored bigraph $H$, the \emph{homomorphism density} of $H$ in $W$ is the
    unique point $t(H,W)\in\RR_+$ in the image of $t((H,\varnothing),W)$ (where $\varnothing$
    denotes the empty function $[0]\to V(H)$), that is, we have
    \begin{equation*}
      t(H,W)
      \df
      \int_{X^{V_1(H)}\times Y^{V_2(H)}}
      \prod_{(u,v)\in E(H)} W_{c(u,v)}(x_u,y_v)
      \ d(\mu\otimes\nu)(x,y).
    \end{equation*}

    Finally, for a finite bigraph $G$ and a bigraphon $W\colon\Omega\times\Lambda\to\RR_+$, the
    \emph{homomorphism density} of $G$ in $W$ is defined as $t(G,W)\df t((G,1),(W)_{i\in [1]})$,
    where $1$ denotes the unique function $E(G)\to[1]$, that is, we have
    \begin{equation*}
      t(G,W)
      \df
      \int_{X^{V_1(G)}\times Y^{V_2(G)}}
      \prod_{(u,v)\in E(G)} W(x_u,y_v)
      \ d(\mu\otimes\nu)(x,y).
    \end{equation*}

  \item Given a finite bigraph $G$ with $v_1(G),v_2(G) > 0$, the \emph{natural bigraphon
    corresponding to $G$} is the bigraphon $W^G\colon V_1(G)\times V_2(G)\to\RR_+$, where $V_1(G)$
    and $V_2(G)$ are equipped with the uniform probability measures and
    \begin{equation*}
      W^G(u,v) \df \One[(u,v)\in E(G)].
    \end{equation*}
    Clearly, if $G'$ is a finite bigraph, then
    \begin{equation*}
      t(G',W^G) = \frac{\lvert\Hom(G',G)\rvert}{v_1(G)^{v_1(G')}\cdot v_2(G)^{v_2(G')}}.
    \end{equation*}

    Given a finite colored bigraph $H$ with $v_1(H),v_2(H) > 0$, the \emph{natural bigraphon
      sequence corresponding to $H$} is the sequence $W^H\df(W^{H_i})_{i\in C_H}$ (that is, the
    $i$th bigraphon is the natural bigraphon of the bigraph $H_i$ consisting only of the edges of
    color $i$). Clearly, if $F$ is a finite flag and $\theta\colon[k_F]\to V(H)$ is injective, then
    \begin{equation}\label{eq:tFWH}
      t(F,W^H)(x,y)
      =
      \frac{
        \lvert\Hom(F,(H,\theta))\rvert
      }{
        v_1(H)^{v_1(F)-\lvert T_1(F)\rvert}\cdot v_2(H)^{v_2(F)-\lvert T_2(F)\rvert}
      },
    \end{equation}
    where
    \begin{align*}
      x_u & \df (\theta\comp\theta_F^{-1})(u), &
      y_v & \df (\theta\comp\theta_F^{-1})(v).
    \end{align*}
  \end{enumdef}
\end{definition}

\begin{definition}[Binary trees]
  \begin{enumdef}
  \item The set $[2]^{<\omega}$ of finite strings over $[2]$ is defined as the union
    $[2]^{<\omega}\df\bigcup_{n\in\NN}[2]^{[n]}$. The \emph{length} of $\eta\in [2]^{<\omega}$ is
    denoted $\lvert\eta\rvert$ (i.e., $\lvert\eta\rvert$ is the unique $n\in\NN$ such that
    $\eta\in [2]^{[n]}$).

    We denote the \emph{prefix partial order} on $[2]^{<\omega}$ by $\preceq$, that is, we have
    \begin{equation*}
      \eta\preceq\tau
      \iff
      \lvert\eta\rvert\leq\lvert\tau\rvert
      \land
      \tau\rest_{[\lvert\eta\rvert]}=\eta.
    \end{equation*}

  \item A \emph{binary tree} is a set $T\subseteq [2]^{<\omega}$ that is $\preceq$-downward closed
    (i.e., if $\eta\in T$ and $\tau\in [2]^{<\omega}$ is such that $\tau\preceq\eta$, then
    $\tau\in T$).

    A \emph{leaf} of a binary tree $T$ is a $\preceq$-maximal element of $T$ (i.e., an element
    $\eta\in T$ such that if $\tau\in T$ is such that $\eta\preceq\tau$, then $\eta=\tau$). An
    \emph{internal node} of $T$ is an element of $T$ that is not a leaf of $T$. The set of leaves of
    $T$ is denoted $L(T)$.

    The \emph{height} of a binary tree $T$ is defined as
    \begin{equation*}
      h(T)\df \sup\{\lvert\eta\rvert \mid \eta\in T\}.
    \end{equation*}
    (Note that $h(T)\in\NN\cup\{-\infty,\infty\}$.)

  \item For a binary tree $T\subseteq [2]^{<\omega}$ and $n\in\NN$, the \emph{$n$-restriction of
    $T$} is the binary tree
    \begin{equation*}
      T\rest_n \df T\cap\bigcup_{m=0}^n [2]^{[n]}.
    \end{equation*}
    Clearly $h(T\rest_n)=\min\{h(T),n\}$.

    The \emph{$1$-extension of $T$} is the tree
    \begin{equation*}
      T\up
      \df
      T\up^1
      \df
      \{\varnothing\}\cup\{\eta\in [2]^{<\omega} \mid \eta\rest_{[\lvert\eta\rvert-1]}\in T\}
    \end{equation*}
    obtained from $T$ by extending it to one more level. (Note that $\varnothing\up =
    \{\varnothing\}$, i.e., the $1$-extension of the empty tree is the tree of height $0$.)

  \item For a finite non-empty tree $T$, the \emph{Cantor probability measure} of $T$ is the
    probability measure $\mu_T$ on $L(T)$, equipped with the discrete $\sigma$-algebra, given by
    \begin{align*}
      \mu_T(\{\eta\}) = 2^{-\lvert\eta\rvert}
    \end{align*}
    for every $\eta\in L(T)$.
  \end{enumdef}
\end{definition}

\begin{definition}[Cut-involutions and folds]
  Let $G$ be a bigraph.
  \begin{enumdef}
  \item A \emph{cut-involution} of $G$ is an automorphism $f$ of $G$ such that:
    \begin{enumerate}
    \item $f$ is an \emph{involution}, that is, $f\comp f = \id_{V(G)}$.
    \item The set
      \begin{equation*}
        \Fix(f) \df \{v\in V(G) \mid f(v) = v\}
      \end{equation*}
      of points that are fixed by $f$ is a cut in $G$.
    \end{enumerate}

    The cut-involution $f$ is further called \emph{independent}\footnote{In~\cite{CL17}, these were
      called \emph{stable involutions}, but to avoid conflict with the notion of fold-stability
      defined later, we use the name independent cut-involution instead.} if $\Fix(f)$ is also an
    independent set of $G$.

  \item A \emph{fold} of $G$ is a pair $(f,L)$ such that:
    \begin{enumerate}
    \item $f$ is a cut-involution of $G$.
    \item $L$ is a union of connected components of $G-\Fix(f)$.
    \item $(L,\Fix(f),f(L))$ is a partition of $V(G)$.
    \end{enumerate}
    The fold $(f,L)$ is further called \emph{independent} if $f$ is further an independent
    cut-involution of $G$.

    The set of folds of $G$ is denoted $\Fold(G)$ and the set of independent folds of $G$ is denoted
    $\IndFold(G)$.

    Note that if $(f,L)\in\Fold(G)$ is a fold of $G$, then so is $(f,f(L))$; this is called the
    \emph{dual fold} of $(f,L)$ (clearly, the dual fold of the dual fold of $(f,L)$ is $(f,L)$
    itself).

    Given a fold $(f,L)\in\Fold(G)$, the \emph{left-folding map} and the \emph{right-folding map} of
    $(f,L)$ are the endomorphisms $f_L,f_L^*\colon V(G)\to V(G)$ of $G$, respectively, defined by
    \begin{align*}
      f_L(v) & \df
      \begin{dcases*}
        f(v), & if $v\in\Fix(f)\cup f(L)$,\\
        v, & if $v\in L$,
      \end{dcases*}
      &
      f_L^*(v) & \df
      \begin{dcases*}
        f(v), & if $v\in L$,\\
        v, & if $v\in\Fix(f)\cup f(L)$.
      \end{dcases*}
    \end{align*}
    Clearly, $f_{f(L)}=f^*_L$ and $f_{f(L)}^*=f_L$.

  \item A \emph{fold} of a colored bigraph $H$ is a fold $(f,L)$ of $G(H)$ such that $f$ is an
    automorphism of $H$. The set of folds of $H$ is denoted $\Fold(H)$

  \item A \emph{coloring tree} of $G$ by colors in $C$ is a function $\psi\colon T\to C^{E(G)}$ from
    a binary tree $T$ to the set of colorings of $G$ by colors in $C$.

    If $T$ is finite and non-empty, then the Cantor probability measure of $\psi$ is the probability
    measure $\mu_\psi$ on $C^{E(G)}$, equipped with the discrete $\sigma$-algebra that is the
    pushforward of $\mu_T$ by $\psi\rest_{L(T)}$, that is, we have
    \begin{equation*}
      \mu_\psi(\{c\})
      \df
      (\psi\rest_{L(T)})_*(\mu_T)(\{c\})
      \df
      \mu_T(\psi\rest_{L(T)}^{-1}(c))
      =
      \sum_{\substack{\eta\in L(T)\\\psi(\eta)=c}} \mu_T(\{\eta\})
      =
      \sum_{\substack{\eta\in L(T)\\\psi(\eta)=c}} 2^{-\lvert\eta\rvert}.
    \end{equation*}

    Given further a sequence $W=(W_i)_{i\in C}$ of bigraphons $W_i\colon\Omega\times\Lambda\to\RR_+$,
    the \emph{homomorphism density} of $\psi$ on $W$ is
    \begin{equation*}
      t(\psi,W)
      \df
      \prod_{c\in C^{E(G)}} t((G,c),W)^{\mu_\psi(\{c\})}.
    \end{equation*}
    Note that the product above has only finitely many terms that are not equal to
    $1$.

  \item A \emph{folding tree of $G$} is a map $\phi\colon T\to\Fold(G)$ from a binary tree $T$ to
    the set of folds of $G$. If $\phi(\eta)=(f,L)$, then we let
    \begin{align*}
      \phi_\eta & \df f_L, &
      \phi_\eta^* & \df f_L^*
    \end{align*}
    be the left-folding and right-folding maps corresponding to the fold $\phi(\eta)=(f,L)$.

    Given $n\in\NN$, the \emph{$n$-restriction} of $\phi$ is $\phi\rest_n\df\phi\rest_{T\rest_n}$.

    Given further a coloring $c\colon E(G)\to C$ of $G$, the \emph{coloring tree induced by $\phi$
    on $c$} is the function $\phi[c]\colon T\up\to C^{E(G)}$  from the $1$-extension of $T$ to the
    set of colorings of $G$ into colors in $C$ given recursively by
    \begin{equation*}
      \phi[c](\eta) \df 
      \begin{dcases*}
        c, & if $\eta = \varnothing$,
        \\
        \phi[c](\eta\rest_{[n-1]})\comp\phi_{\eta\rest_{[n-1]}},
        & if $\lvert\eta\rvert=n>0$ and $\eta_n=1$,
        \\
        \phi[c](\eta\rest_{[n-1]})\comp\phi_{\eta\rest_{[n-1]}}^*,
        & if $\lvert\eta\rvert=n>0$ and $\eta_n=2$,
      \end{dcases*}
    \end{equation*}

  \item Given a set $C$ of colors and a set $\cF\subseteq\Fold(G)$ of folds of $G$, we define the
    (partial) preorder $\preceq_{\cF}$ on $C^{E(G)}$ by letting $c\preceq_{\cF} c'$ when there
    exists a folding tree $\phi$ of $G$ such that $\im(\phi)\subseteq\cF$ and
    $c'\in\im(\phi[c])$. Equivalently, when there exist $t\in\NN$ and folds
    $(f_1,L_1),\ldots,(f_t,L_t)$ such that for every $i\in[t]$, either $(f_i,L_i)$ or its dual
    $(f_i,f_i(L_i))$ is in $\cF$ and
    \begin{equation*}
      c' = c\comp (f_1)_{L_1}\comp\cdots\comp (f_t)_{L_t}.
    \end{equation*}

    When $c\preceq_{\cF} c'$, we say that $c'$ is \emph{$\cF$-reachable} from $c$. Clearly the
    notion of $\cF$-reachability does not change if we close $\cF$ under dual folds.

    An element $c\in C^{E(G)}$ is called \emph{$\preceq_{\cF}$-maximal} if for every $c'\in
    C^{E(G)}$, if $c\preceq_{\cF} c'$, then $c'\preceq_{\cF} c$.

  \item A subset $\cC\subseteq C^{E(G)}$ is called \emph{$\preceq_{\cF}$-strongly connected} if for
    every $c,c'\in\cC$, we have $c\preceq_{\cF} c'$.

    A \emph{$\preceq_{\cF}$-strongly connected component} of $C^{E(G)}$ is a maximal
    $\preceq_{\cF}$-strongly connected non-empty subset of $C^{E(G)}$, that is, a non-empty subset
    of $C^{E(G)}$ that is $\preceq_{\cF}$-strongly connected and is not properly contained in any
    other $\preceq_{\cF}$-strongly connected subset of $C^{E(G)}$.
  \end{enumdef}
\end{definition}

\begin{definition}[Folding problems]
  \begin{enumdef}
  \item A \emph{folding problem} is a tuple $(G,c,\cC,\cF)$, where:
    \begin{itemize}
    \item $G$ is a bigraph, called the \emph{base bigraph}.
    \item $c$ is a coloring of $G$, called the \emph{initial coloring}.
    \item $\cC\subseteq C_c^{E(G)}$ is a set of colorings of $G$ by $C_c$, called the \emph{set of
      objective colorings}.
    \item $\cF\subseteq\Fold(G)$ is a set of folds of $G$, called the \emph{set of allowed folds}.
    \end{itemize}

    We say that $(G,c,\cC,\cF)$ is \emph{finite} if $G$ is finite.

  \item We say that a folding problem $(G,c,\cC,\cF)$ is a \emph{positive reachability instance}
    if there exists $c'\in\cC$ such that $c\preceq_{\cF} c'$. We abbreviate this as $c\preceq_{\cF}\cC$.

    We say $(G,c,\cC,\cF)$ is a \emph{negative reachability instance} when it is not a positive
    reachability instance. We abbreviate this as $c\not\preceq_{\cF}\cC$.

    Clearly, these notions do not change if we close $\cF$ under dual folds.

  \item We say that a folding problem $(G,c,\cC,\cF)$ is a \emph{positive percolation instance} if
    there exists a sequence of folding trees $(\phi_n)_{n\in\NN}$ of $G$ with $\phi_n\colon
    T_n\to\Fold(G)$ such that:
    \begin{enumerate}
    \item For every $n\in\NN$, we have $\im(\phi_n)\subseteq\cF$.
    \item For every $n\in\NN$, we have $T_n\subseteq T_{n+1}$ and $\phi_{n+1}\rest_{T_n}=\phi_n$.
    \item We have
      \begin{equation*}
        \lim_{n\to\infty} \mu_{\phi_n[c]}(\cC) = 1.
      \end{equation*}
    \end{enumerate}
    We abbreviate this as $c\ll_{\cF}\cC$ and we say that such a sequence $(\phi_n)_{n\in\NN}$
    \emph{witnesses} $c\ll_{\cF}\cC$.

    We say that $(G,c,\cC,\cF)$ is a \emph{negative percolation instance} when it is not a positive
    percolation instance. We abbreviate this as $c\not\ll_{\cF}\cC$.

    Clearly, these notions do not change if we close $\cF$ under dual folds.
    
  \item We say that a folding problem $(G,c,\cC,\cF)$ is \emph{absorbing} if for every $c'\in
    C_c^{E(G)}$ with $c\preceq_{\cF} c'$ and every fold $(f,L)\in\cF$, if $c'\preceq_{\cF}\cC$ and
    $c'\notin\cC$, then $c'\comp f_L\preceq_{\cF}\cC$ and $c'\comp f_L^*\preceq_{\cF}\cC$.

  \item Let $G$ be a bigraph and $\cF\subseteq\Fold(G)$ be a set of folds of $G$. We say that $G$ is
    \emph{$\cF$-cut-percolating} if there exists a finite sequence $E_0,E_1,\ldots,E_n\subseteq
    E(G)$ such that:
    \begin{enumerate}
    \item We have $\lvert E_0\rvert=1$.
    \item For every $i\in[n]$, there exists a fold $(f,L)\in\cF$ such that $E_i=f_L^{-1}(E_{i-1})$.
    \item We have $E_n=E(G)$.
    \end{enumerate}
    We drop $\cF$ from the notation when $\cF=\Fold(G)$.

    The bigraph $G$ is called \emph{independently cut-percolating} if it is
    $\IndFold(G)$-cut-percolating.

  \item Let $H$ be a colored bigraph and $\cF\subseteq\Fold(H)$ be a set of folds of $H$. We say
    that $H$ is \emph{$\cF$-left-cut-percolating} if there exists a finite sequence
    $V_0,V_1,\ldots,V_n\subseteq V_1(H)$ such that:
    \begin{enumerate}
    \item We have $\lvert V_0\rvert=1$.
    \item For every $i\in[n]$, there exists a fold $(f,L)\in\cF$ such that $V_i=f_L^{-1}(V_{i-1})$.
    \item We have $V_n=V_1(H)$.
    \end{enumerate}
    We drop $\cF$ from the notation when $\cF=\Fold(H)$.
  \end{enumdef}
\end{definition}

\begin{definition}[Fold-stability]
  Let $G$ be a bigraph, let $c\colon E(G)\to C$ be a coloring of $G$, let
  $\cF,\cF'\subseteq\Fold(G)$ be sets of folds of $G$ and let $K$ be a subgroup of the automorphism
  group $\Aut(G)$ of $G$. We say that $c$ is:
  \begin{description}
  \item[$(K,\cF)$-fold-stable] if for every fold $(f,L)\in\cF$, there exists $g\in K$ that is an
    isomorphism from $(G,c)$ to $(G,c\comp f_L)$.
  \item[$\cF'$-strongly $(K,\cF)$-fold-stable] if it is $(K,\cF)$-fold-stable\footnote{Let us point
    out that in Lemma~\ref{lem:foldstable}\ref{lem:foldstable:strong}, we will show that the
    existence of $h$ as required by the second part of the definition immediately implies
    $(\Aut(G),\cF)$-fold-stability.} and for every fold $(f,L)\in\cF$ there exists an automorphism
    $h\in\Aut((G,c))$ of the colored bigraph $(G,c)$ such that $\Fix(h)=\Fix(f)$ and $(h,L)\in\cF'$.
  \item[$\cF'$-symmetrically $(K,\cF)$-fold-stable] if it is $\cF'$-strongly $(K,\cF)$-fold-stable
    and for every $i,j\in\im(c)$, there exist $g\in K$ and a permutation $\sigma\colon C_c\to C_c$
    such that $\sigma(i)=j$ and $g$ is an isomorphism from $(G,c)$ to $(G,\sigma\comp c)$.
  \end{description}

  In all of the above, we drop $\cF$ or $\cF'$ from the notation when they are equal to $\Fold(G)$
  and we drop $K$ from the notation when $K=\Aut(G)$.
\end{definition}

\begin{definition}[Left-sided notions]
  Let $G$ be a bigraph.
  \begin{enumdef}
  \item A \emph{left-coloring} of $G$ is a function $\ell\colon V_1(G)\to C$, where $C$ is any
    set. We use the shorthand notation $C_\ell\df C$. We say that $\ell$ is:
    \begin{description}
    \item[monochromatic] if $\lvert\im(\ell)\rvert\leq 1$.
    \item[rainbow] if $\ell$ is injective.
    \end{description}

  \item If $\ell\colon V_1(G)\to C$ is a left-coloring of $G$ and $c\colon E(G)\to C'$ is a coloring
    of $G$, then we define the coloring $\ell\otimes c\colon E(G)\to C\times C'$ by
    \begin{equation*}
      (\ell\otimes c)(u,v) \df (\ell(u), c(u,v)).
    \end{equation*}

    It is worth noting that folding maps respect this product in the following sense: if
    $(f,L)\in\Fold(G)$, then
    \begin{align*}
      (\ell\otimes c)\comp f_L & = (\ell\comp f_L)\otimes(c\comp f_L), &
      (\ell\otimes c)\comp f_L^* & = (\ell\comp f_L^*)\otimes(c\comp f_L^*).
    \end{align*}
    In particular, if $(f,L)\in\Fold((G,c))$ is a fold of the colored bigraph $(G,c)$, then the
    above simplifies to
    \begin{align*}
      (\ell\otimes c)\comp f_L & = (\ell\comp f_L)\otimes c, &
      (\ell\otimes c)\comp f_L^* & = (\ell\comp f_L^*)\otimes c.
    \end{align*}

  \item For a left-coloring $\ell\colon V_1(G)\to C$ of $G$, a coloring $c\colon E(G)\to C'$ of $G$,
    sets of folds $\cF,\cF'\subseteq\Fold(G)$ of $G$ and a subgroup $K$ of $\Aut(G)$, we say that
    $\ell\otimes c$ is \emph{$\cF'$-left-symmetrically $(K,\cF)$-fold-stable} if it is
    $\cF'$-strongly $(K,\cF)$-fold-stable and for every $i,j\in\im(\ell)$, there exist $g\in K$ and
    a permutation $\sigma\colon C\to C$ such that $\sigma(i)=j$ and $g$ is an isomorphism from
    $(G,\ell\otimes c)$ to $(G,(\sigma\comp\ell)\otimes c)$.

    We drop $\cF$ from the notation when $\cF=\Fold(G)$ and we drop $K$ from the notation when
    $K=\Aut(G)$.
  \end{enumdef}
\end{definition}

\section{Basic properties of folds}
\label{sec:basicprops}

In this section, we prove properties about folds that follow almost immediately from definitions,
but are nevertheless needed for later results. The first two results have previously appeared in the
literature, but since their proofs are short, we include them here.

\begin{lemma}[Conlon--Lee~\cite{CL17}]\label{lem:Kedgetransitive}
  Let $G$ be a bigraph, let $\cF\subseteq\Fold(G)$ be a set of folds of $G$ and let $K$ be the
  subgroup of $\Aut(G)$ generated by $\{f\mid (f,L)\in\cF\}$. If $G$ is $\cF$-cut-percolating, then
  $G$ is $K$-edge-transitive.
\end{lemma}

\begin{proof}
  Since $G$ is $\cF$-cut-percolating, there exist a finite sequence $E_0,E_1,\ldots,E_n\subseteq
  E(G)$ and a sequence of folds $(f_1,L_1),\ldots,(f_n,L_n)\in\cF$ such that $\lvert E_0\rvert=1$,
  $E_n=E(G)$ and $E_i=(f_i)_{L_i}^{-1}(E_{i-1})$ for every $i\in[n]$.

  Let $e_0$ be the unique element of $E_0$ and note that since for every $e\in E(G)$, we have
  \begin{equation*}
    e_0 = ((f_1)_{L_1}\comp(f_2)_{L_2}\comp\cdots\comp (f_n)_{L_n})(e),
  \end{equation*}
  there must exist indices $i_1^e < \cdots < i_{t_e}^e$ such that
  \begin{equation*}
    e_0 = (f_{i_1^e}\comp f_{i_2^e}\comp\cdots\comp f_{i_{t_e}^e})(e),
  \end{equation*}
  from which we conclude that $G$ is $K$-edge-transitive.
\end{proof}

\begin{lemma}[Coregliano~\cite{Cor24}]\label{lem:Kleftvertextransitive}
  Let $H$ be a colored bigraph, let $\cF\subseteq\Fold(H)$ be a set of folds of $H$ and let $K$ be
  the subgroup of $\Aut(H)$ generated by $\{f\mid (f,L)\in\cF\}$. If $H$ is
  $\cF$-left-cut-percolating, then $H$ is $K$-left-vertex-transitive.
\end{lemma}

\begin{proof}
  Since $H$ is $\cF$-left-cut-percolating, there exist a finite sequence $V_0,V_1,\ldots,V_n\subseteq
  V_1(H)$ and a sequence of folds $(f_1,L_1),\ldots,(f_n,L_n)\in\cF$ such that $\lvert V_0\rvert=1$,
  $V_n=V_1(H)$ and $V_i=(f_i)_{L_i}^{-1}(V_{i-1})$ for every $i\in[n]$.

  Let $v_0$ be the unique element of $V_0$ and note that since for every $v\in V_1(H)$, we have
  \begin{equation*}
    v_0 = ((f_1)_{L_1}\comp(f_2)_{L_2}\comp\cdots\comp (f_n)_{L_n})(v),
  \end{equation*}
  there must exist indices $i_1^v < \cdots < i_{t_v}^v$ such that
  \begin{equation*}
    v_0 = (f_{i_1^v}\comp f_{i_2^v}\comp\cdots\comp f_{i_{t_v}^v})(v),
  \end{equation*}
  from which we conclude that $H$ is $K$-left-vertex-transitive.
\end{proof}

\begin{lemma}\label{lem:action}
  If $(f,L)\in\Fold(G)$ is a fold of a bigraph $G$ and $h\in\Aut(G)$ is an automorphism of $G$, then
  $(h\comp f\comp h^{-1}, h(L))$ is a fold of $G$ with
  \begin{align*}
    \Fix(h\comp f\comp h^{-1}) & = h(\Fix(f)), &
    (h\comp f\comp h^{-1})_{h(L)} & = h\comp f_L\comp h^{-1}, &
    (h\comp f\comp h^{-1})_{h(L)}^* & = h\comp f_L^*\comp h^{-1}.
  \end{align*}

  In particular, this yields a left action of $\Aut(G)$ on $\Fold(G)$ given by
  \begin{equation}\label{eq:AutGFoldG}
    h\cdot (f,L) \df (h\comp f\comp h^{-1}, h(L))
  \end{equation}
  for every $h\in\Aut(G)$ and every $(f,L)\in\Fold(G)$.
\end{lemma}

\begin{proof}
  It is straightforward to check that $h\comp f\comp h^{-1}$ is an involution of $G$ and
  $\Fix(h\comp f\comp h^{-1}) = h(\Fix(f))$. Since $L$ is a union of connected components of
  $G-\Fix(f)$, it follows that $h(L)$ is a union of connected components of $G-h(\Fix(f)) =
  G-\Fix(h\comp f\comp h^{-1})$. Finally, since $(L,\Fix(f),f(L))$ is a partition of $V(G)$, it
  follows that
  \begin{equation*}
    (h(L),\Fix(h\comp f\comp h^{-1}),(h\comp f\comp h^{-1})(h(L)))
    =
    (h(L),h(\Fix(f)),h(f(L)))
  \end{equation*}
  is also a partition of $V(G)$. Thus, $(h\comp f\comp h^{-1}, h(L))\in\Fold(G)$.

  Finally, it is straightforward to check that~\eqref{eq:AutGFoldG} indeed is an action of $\Aut(G)$
  on $\Fold(G)$.
\end{proof}

\begin{lemma}\label{lem:indfold}
  The set $\IndFold(G)$ of independent folds of a bigraph $G$ is $\Aut(G)$-invariant.
\end{lemma}

\begin{proof}
  Follows since for every $h\in\Aut(G)$ and every independent fold $(f,L)\in\IndFold(G)$ of $G$, by
  Lemma~\ref{lem:action}, the fold $(f',L')\df h\cdot (f,L)$ satisfies $\Fix(f') = h(\Fix(f))$,
  which is an independent set of $G$ as $\Fix(f)$ is so and $h$ is an automorphism of $G$.
\end{proof}

\begin{lemma}\label{lem:dualfold}
  Let $G$ be a bigraph, let $\cF\subseteq\Fold(G)$ be a set of folds of $G$ and let $K$ be the
  subgroup of $\Aut(G)$ generated by $\{f\mid (f,L)\in\cF\}$. If $\cF$ is $K$-invariant, then $\cF$
  is closed under dual folds.
\end{lemma}

\begin{proof}
  Follows from Lemma~\ref{lem:action} and since for $(f,L)\in\cF$, its dual fold can be written as
  \begin{equation*}
    (f,f(L)) = (f\comp f\comp f^{-1}, f(L)) = f\cdot (f,L).\qedhere
  \end{equation*}
\end{proof}

\section{Flag isomorphism versus homomorphism densities}
\label{sec:flagiso}

One of the classic results of graphon theory by \Lovasz\ (see~\cite[Theorem~5.29]{Lov12}) is that
two bigraphs $G_1$ and $G_2$ without isolated vertices are isomorphic if and only if
$t(G_1,W)=t(G_2,W)$ for every bigraphon $W$. The same argument easily applies to extend this result
to flags (see Proposition~\ref{prop:flagiso} below). The main goal of this section is to answer
instead a related question: when are the functions $t(F_1,W)$ and $t(F_2,W)$ linearly dependent for
every bigraphon $W$? As we will see in Proposition~\ref{prop:coreiso}, this happens precisely when
the restrictions to the connected cores, $F_1\rest_{C(F_1)}$ and $F_2\rest_{C(F_2)}$, are
isomorphic.

\begin{lemma}\label{lem:flagiso}
  Let $F_1$ and $F_2$ be finite flags and let $W = (W_i)_{i\in C_{F_1}}$ be a sequence of bigraphons
  $W_i\colon\Omega\times\Lambda\to\RR_+$. If $F_1\cong F_2$, then
  \begin{equation*}
    t(F_1,W)\comp\eta = t(F_2,W),
  \end{equation*}
  where $\eta\colon X^{T_1(F_2)}\times Y^{T_2(F_2)}\to X^{T_1(F_1)}\times Y^{T_2(F_1)}$ is given by
  $\eta(x,y)\df (x',y')$, where
  \begin{align*}
    x'_u & \df x_{(\theta_{F_2}\comp\theta_{F_1}^{-1})(u)}, &
    y'_v & \df y_{(\theta_{F_2}\comp\theta_{F_1}^{-1})(v)}.
  \end{align*}
\end{lemma}

\begin{proof}
  Follows by relabeling the variables of integration through the isomorphism.
\end{proof}

\begin{proposition}\label{prop:flagiso}
  Let $F_1$ and $F_2$ be finite flags with the same type and suppose that for every $t\in[2]$, all
  isolated vertices of $G(F_t)$ are in $T(F_t)$. Then following are equivalent.
  \begin{enumerate}
  \item\label{prop:flagiso:iso} We have $F_1\cong F_2$.
  \item\label{prop:flagiso:hom} For every sequence $W = (W_i)_{i\in C_{F_1}}$ of bigraphons
    $W_i\colon\Omega\times\Lambda\to\RR_+$ and every $(x,y)\in X^{T_1(F)}\times Y^{T_2(F)}$, we have
    \begin{equation*}
      t(F_1,W)\comp\eta = t(F_2,W),
    \end{equation*}
    where $\eta\colon X^{T_1(F_2)}\times Y^{T_2(F_2)}\to X^{T_1(F_1)}\times Y^{T_2(F_1)}$ is given by
    $\eta(x,y)\df (x',y')$, where
    \begin{align*}
      x'_u & \df x_{(\theta_{F_2}\comp\theta_{F_1}^{-1})(u)}, &
      y'_v & \df y_{(\theta_{F_2}\comp\theta_{F_1}^{-1})(v)}.
    \end{align*}
  \item\label{prop:flagiso:homae} Item~\ref{prop:flagiso:hom} holds for $\mu\otimes\nu$-almost every
    $(x,y)$ in place of every $(x,y)$.
  \item\label{prop:flagiso:homcount} For every finite flag $F$, we have
    $\lvert\Hom(F_1,F)\rvert=\lvert\Hom(F_2,F)\rvert$.
  \end{enumerate}
\end{proposition}

\begin{proof}
  The implication~\ref{prop:flagiso:iso}$\implies$\ref{prop:flagiso:hom} follows directly from
  Lemma~\ref{lem:flagiso} and the
  implication~\ref{prop:flagiso:hom}$\implies$\ref{prop:flagiso:homae} is trivial.

  \medskip

  We now prove the implication~\ref{prop:flagiso:homae}$\implies$\ref{prop:flagiso:homcount}.

  We start by showing that $v_j(F_1)=v_j(F_2)$ for every $j\in[2]$. We show only the case $j=1$ as
  the other case is analogous. For this, consider the bigraphon $U\colon[2]\times[1]\to\RR_+$ given
  by $U(x,y)\df\One[x=1]$, where both $[2]$ and $[1]$ are equipped with uniform probability
  measures and we consider the sequence $W=(U)_{i\in C_{F_1}}$ of bigraphons that is constant equal
  to $U$. Then by letting $(x,y)\in [2]^{T_1(F_2)}\times[1]^{T_2(F_2)}$ be given by
  \begin{equation*}
    x_u \df y_v \df 1,
  \end{equation*}
  we have
  \begin{align*}
    t(F_1,W)(\eta(x,y))
    & =
    2^{-v_1(F_1)+\lvert T_1(F_1)\rvert},
    &
    t(F_2,W)(x,y)
    & =
    2^{-v_1(F_2)+\lvert T_1(F_2)\rvert}.
  \end{align*}

  Since $t(F_1,W)(\eta(x,y))=t(F_2,W)(x,y)$ (as $t(F_1,W)\comp\eta = t(F_2,W)$ a.e.) and $\lvert
  T_1(F_1)\rvert = \lvert T_1(F_2)\rvert$ (as $F_1$ and $F_2$ have the same type), it follows that
  $v_1(F_1)=v_1(F_2)$.
  
  We now show that $\lvert\Hom(F_1,F)\rvert=\lvert\Hom(F_2,F)\rvert$ for every finite flag $F$. Note
  that if $F$ does not have the same type as $F_1$ (and $F_2$), then both $\Hom(F_1,F)$ and
  $\Hom(F_2,F)$ are empty. Suppose then that $F$ has the same type as $F_1$ (and $F_2$).

  If $v_j(F)=0$ for some $j\in[2]$, then $\Hom(F_1,F)$ and $\Hom(F_2,F)$ are either both empty (if
  $v_j(F_1)=v_j(F_2) > 0$) or contain only the function that show that $F_1$, $F_2$ and $F$ all have
  the same type (as we must have $v_j(F_1)=v_j(F_2)=0$, which forces $F_1$ and $F_2$ to have no
  edges, hence all their vertices must be in $T(F_1)$ and $T(F_2)$, respectively, as they are
  isolated).

  Suppose then that $v_1(F),v_2(F)>0$. By~\eqref{eq:tFWH}, for each $t\in[2]$, we have
  \begin{equation*}
    \lvert\Hom(F_t,F)\rvert
    =
    v_1(F)^{v_1(F_t)-\lvert T_1(F)\rvert}\cdot v_2(F)^{v_2(F_t)-\lvert T_2(F)\rvert}
    \cdot t(F_t,W^{H(F)})(x^t,y^t),
  \end{equation*}
  where
  \begin{align*}
    x^t_u & \df (\theta_F\comp\theta_{F_t}^{-1})(u), &
    y^t_u & \df (\theta_F\comp\theta_{F_t}^{-1})(v).
  \end{align*}

  We have already shown that $v_j(F_1)=v_j(F_2)$ ($j\in[2]$) and we know that $\lvert
  T_j(F_1)\rvert=\lvert T_j(F_2)\rvert$ ($j\in[2]$) since $F_1$ and $F_2$ have the same type, so it
  suffices to show $t(F_1,W^{H(F)})(x^1,y^1) = t(F_2,W^{H(F)})(x^2,y^2)$. But indeed, this follows
  since
  \begin{equation*}
    t(F_1,W^{H(F)})(x^1,y^1)
    =
    t(F_1,W^{H(F)})(\eta(x^2,y^2))
    =
    t(F_2,W^{H(F)})(x^2,y^2),
  \end{equation*}
  where the second equality follows from the almost everywhere assumption of
  item~\ref{prop:flagiso:homae} (as $(x^2,y^2)$ is an atom of the space).

  Thus $\lvert\Hom(F_1,F)\rvert=\lvert\Hom(F_2,F)\rvert$ for every finite flag $F$.

  \medskip
  
  For the final implication~\ref{prop:flagiso:homcount}$\implies$\ref{prop:flagiso:iso}, let
  \begin{equation*}
    m \df \lvert\{g\in\Hom(F_1,F_2) \mid \im(g) = V(F_2)\}\rvert
  \end{equation*}
  be the number of surjective homomorphisms from $F_1$ to $F_2$. By inclusion-exclusion (twice), we
  have
  \begin{align*}
    m
    & =
    \sum_{\substack{U\subseteq V(F_2)\\ T(F_2)\subseteq U}}
    (-1)^{\lvert V(F_2)\setminus U\rvert}\cdot\lvert\Hom(F_1,F_2\rest_U)\rvert
    \\
    & =
    \sum_{\substack{U\subseteq V(F_2)\\ T(F_2)\subseteq U}}
    (-1)^{\lvert V(F_2)\setminus U\rvert}\cdot\lvert\Hom(F_2,F_2\rest_U)\rvert
    \\
    & =
    \lvert\{g\in\Hom(F_2,F_2) \mid \im(g)=V(F_2)\}\rvert
    \geq
    1,
  \end{align*}
  where the second equality follows from the assumption of item~\ref{prop:flagiso:homcount}.

  Thus, there exists some surjective homomorphism $g_1$ from $F_1$ to $F_2$, which in particular
  implies $v(F_1)\geq v(F_2)$.

  Analogously, it follows that there exists a surjective homomorphism $g_2$ from $F_2$ to $F_1$ and
  $v(F_2)\geq v(F_1)$.

  Thus, we conclude that $v(F_1)=v(F_2)$ and since both these are finite quantities, it follows that
  both $g_1$ and $g_2$ are in fact bijective. Since $g_1$ is an injective homomorphism, it follows
  that $e(F_1)\leq e(F_2)$. Applying the same rationale for $g_2$, we get $e(F_2)\leq e(F_1)$, hence
  $e(F_1)=e(F_2)$. Since $e(F_1)=e(F_2)$ is finite, it follows that $g_1$ (and $g_2$) is in fact an
  embedding, hence an isomorphism.
\end{proof}

\begin{proposition}\label{prop:coreiso}
  Let $F_1$ and $F_2$ be finite flags with the same type. Then following are equivalent.
  \begin{enumerate}
  \item\label{prop:coreiso:iso} We have $F_1\rest_{C(F_1)}\cong F_2\rest_{C(F_2)}$.
  \item\label{prop:coreiso:hom} For every sequence $W = (W_i)_{i\in C_{F_1}}$ of bigraphons
    $W_i\colon\Omega\times\Lambda\to\RR_+$, there exist $\lambda_1,\lambda_2\in\RR_+$ not both zero
    such that for every $(x,y)\in X^{T_1(F)}\times Y^{T_2(F)}$, we have
    \begin{equation*}
      \lambda_1\cdot t(F_1,W)\comp\eta = \lambda_2\cdot t(F_2,W),
    \end{equation*}
    where $\eta\colon X^{T_1(F_2)}\times Y^{T_2(F_2)}\to X^{T_1(F_1)}\times Y^{T_2(F_1)}$ is given by
    $\eta(x,y)\df (x',y')$, where
    \begin{align*}
      x'_u & \df x_{(\theta_{F_2}\comp\theta_{F_1}^{-1})(u)}, &
      y'_v & \df y_{(\theta_{F_2}\comp\theta_{F_1}^{-1})(v)}.
    \end{align*}
  \item\label{prop:coreiso:homae} Item~\ref{prop:coreiso:hom} holds for $\mu\otimes\nu$-almost every
    $(x,y)$ in place of every $(x,y)$.
  \item\label{prop:coreiso:homcore} For every sequence $W = (W_i)_{i\in C_{F_1}}$ of bigraphons
    $W_i\colon\Omega\times\Lambda\to\RR_+$, there exist $\lambda_1,\lambda_2\in\RR_+$ not both zero
    such that for every $(x,y)\in X^{T_1(F)}\times Y^{T_2(F)}$, we have
    \begin{equation*}
      \lambda_1\cdot t(F_1\rest_{C(F_1)},W)\comp\eta = \lambda_2\cdot t(F_2\rest_{C(F_2)},W),
    \end{equation*}
    where $\eta\colon X^{T_1(F_2)}\times Y^{T_2(F_2)}\to X^{T_1(F_1)}\times Y^{T_2(F_1)}$ is given by
    $\eta(x,y)\df (x',y')$, where
    \begin{align*}
      x'_u & \df x_{(\theta_{F_2}\comp\theta_{F_1}^{-1})(u)}, &
      y'_v & \df y_{(\theta_{F_2}\comp\theta_{F_1}^{-1})(v)}.
    \end{align*}
  \item\label{prop:coreiso:homcoreae} Item~\ref{prop:coreiso:homcore} holds for $\mu\otimes\nu$-almost
    every $(x,y)$ in place of every $(x,y)$.
  \item\label{prop:coreiso:homcoreone} Item~\ref{prop:coreiso:homcore} holds with
    $\lambda_1=\lambda_2=1$.
  \item\label{prop:coreiso:homcoreaeone} Item~\ref{prop:coreiso:homcoreae} holds with
    $\lambda_1=\lambda_2=1$.
  \end{enumerate}
\end{proposition}

\begin{proof}
  The structure of the proof is as follows:
  
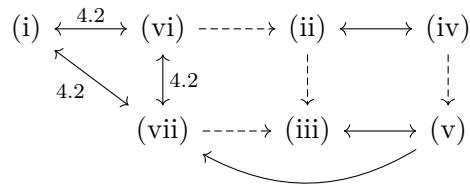
\begin{equation*}
    \begin{tikzcd}
      \text{\ref{prop:coreiso:iso}}
      \arrow[r, leftrightarrow, "\ref{prop:flagiso}"]
      \arrow[dr, leftrightarrow, "\ref{prop:flagiso}"']
      &
      \text{\ref{prop:coreiso:homcoreone}}
      \arrow[d, leftrightarrow, "\ref{prop:flagiso}"]
      \arrow[r, dashed]
      &
      \text{\ref{prop:coreiso:hom}}
      \arrow[d, dashed]
      \arrow[r, leftrightarrow]
      &
      \text{\ref{prop:coreiso:homcore}}
      \arrow[d, dashed]
      \\
      &
      \text{\ref{prop:coreiso:homcoreaeone}}
      \arrow[r, dashed]
      &
      \text{\ref{prop:coreiso:homae}}
      \arrow[r, leftrightarrow]
      &
      \text{\ref{prop:coreiso:homcoreae}}
      \arrow[ll, bend left={30}]
    \end{tikzcd}
  \end{equation*}
  In the above, the arrows labeled~\ref{prop:flagiso} follow directly from
  Proposition~\ref{prop:flagiso} and the dashed arrows are trivial.

  Before we prove the other arrows, note that by appealing to the Dominated Convergence Theorem and
  approximating a bigraphon $W$ by $W_\epsilon\df\max\{W,\epsilon\}$ as $\epsilon\to 0^+$, in all
  items involving bigraphons, it suffices to show only the result for strictly positive
  bigraphons. In turn, strictly positive bigraphons have strictly positive homomorphism density
  functions.

  For the two horizontal double-arrows between the two rightmost columns, note that
  \begin{equation*}
    t(F_t,W)(x,y)
    =
    t(H(F_t) - C(F_t), W)\cdot t(F_t\rest_{C(F_t)},W)(x,y)
  \end{equation*}
  for every $t\in[2]$ and every $(x,y)\in X^{T_1(F_t)}\times Y^{T_2(F_t)}$, that is, the functions
  $t(F_t,W)$ and $t(F_t\rest_{C(F_t)},W)$ are non-negative multiples of each other. By considering
  strictly positive bigraphons, these functions are positive multiples of each other, so the two
  horizontal double-arrows between the two rightmost columns follow.

  It remains to prove the
  implication~\ref{prop:coreiso:homcoreae}$\implies$\ref{prop:coreiso:homcoreaeone}.

  As argued before, we consider only sequences $W$ of bigraphons that are strictly positive. Since
  for one such sequence $W$ both $t(F_1,W)$ and $t(F_2,W)$ are strictly positive functions, we know
  that there must exist $\lambda^W > 0$ such that
  \begin{equation*}
    t(F_1\rest_{C(F_1)},W)\comp\eta^W = \lambda^W\cdot t(F_2\rest_{C(F_2)},W)
  \end{equation*}
  almost everywhere for the corresponding $\eta^W$ function.

  We need to show that $\lambda_W=1$ for every $W$. First, note that for the sequence $\One$ of
  bigraphons that are $1$ everywhere (over some given spaces), both functions $t(F_1,\One)$ and
  $t(F_2,\One)$ are $1$ everywhere, from which we get $\lambda_{\One}=1$.

  Thus, to show that $\lambda_W=1$ for every $W$, it suffices to show that for any two sequences
  $W^1,W^2$ of strictly positive bigraphons $W^j\colon\Omega_j\times\Lambda_j\to(0,\infty)$
  ($j\in[2]$), we have $\lambda_{W^1}=\lambda_{W^2}$.

  For this, we define a new sequence of bigraphons $W=(W_i)_{i\in C_{F_1}}$ by $W_i\colon(X_1\cup
  X_2)\times(Y_1\cup Y_2)\to\RR_+$, where the spaces are equipped with the co-product
  $\sigma$-algebra and the probability measures $\mu$ and $\nu$ given by
  \begin{align*}
    \mu(A) & \df \frac{\mu_1(A\cap X_1) + \mu_2(A\cap X_2)}{2}, &
    \nu(A) & \df \frac{\nu_1(A\cap X_1) + \nu_2(A\cap X_2)}{2},
  \end{align*}
  and $W_i$ is given by
  \begin{equation*}
    W_i(x,y) \df
    \begin{dcases*}
      W^1_i(x,y), & if $x\in X_1$ and $y\in Y_1$,\\
      W^2_i(x,y) & if $x\in X_2$ and $y\in Y_2$,\\
      0, & otherwise.
    \end{dcases*}
  \end{equation*}
  (Note that, even though $W^1$ and $W^2$ are strictly positive, $W$ is not.)

  Since $F_1\rest_{C(F_1)}$ and $F_2\rest_{C(F_2)}$ are core-connected, it follows that for every
  $j\in[2]$ and every $(x,y)\in X_j^{T_1(F_2)}\times Y_j^{T_2(F_2)}$, we have
  \begin{align*}
    t(F_1\rest_{C(F_1)},W)(\eta(x,y))
    & =
    \frac{t(F_1\rest_{C(F_1)},W^j)(\eta(x,y))}{2^{\lvert C(F_1)\setminus T(F_1)\rvert}},
    &
    t(F_2\rest_{C(F_2)},W)(x,y)
    & =
    \frac{t(F_2\rest_{C(F_2)},W^j)(x,y)}{2^{\lvert C(F_2)\setminus T(F_2)\rvert}},
  \end{align*}
  which together with the definition of $\lambda^{W^j}$ yields
  \begin{equation}\label{eq:lambdaWj}
    \lambda_{W^j}
    =
    \frac{t(F_1\rest_{C(F_1)},W)(\eta(x,y))}{t(F_2\rest_{C(F_2)},W)(x,y)}
    \cdot 2^{\lvert C(F_1)\rvert - \lvert C(F_2)\rvert}
  \end{equation}
  for $\mu_j\otimes\nu_j$-almost every $(x,y)\in X_j^{T_1(F_2)}\times Y_j^{T_2(F_2)}$.

  Applying the assumption of item~\ref{prop:coreiso:homcoreae} to $W$, we know that there exist
  $\lambda_1,\lambda_2\in\RR_+$ not both zero such that
  \begin{equation*}
    \lambda_1\cdot t(F_1\rest_{C(F_1)},W)\comp\eta
    =
    \lambda_2\cdot t(F_2\rest_{C(F_2)},W)
  \end{equation*}
  almost everywhere, which together with~\eqref{eq:lambdaWj} yields
  \begin{equation*}
    \lambda_1\cdot\lambda_{W^j}
    =
    \lambda_2\cdot 2^{\lvert C(F_1)\rvert - \lvert C(F_2)\rvert}
  \end{equation*}
  for each $j\in[2]$. Since $t(F_2\rest_{C(F_2)},W) > 0$, we conclude that $\lambda_1 > 0$, from
  which we get
  \begin{equation*}
    \lambda_{W^j}
    =
    \frac{\lambda_2}{\lambda_1}\cdot 2^{\lvert C(F_1)\rvert - \lvert C(F_2)\rvert}
  \end{equation*}
  and since this holds for both $j=1$ and $j=2$, we get $\lambda^{W^1}=\lambda^{W^2}$.

  This concludes the proof of the proposition.
\end{proof}

\section{Reachability versus percolation}
\label{sec:reachvsperc}

It is obvious that if a folding problem $(G,c,\cC,\cF)$ is a positive percolation instance then it
must be a positive reachability instance (i.e., $c\ll_{\cF}\cC$ implies $c\preceq_{\cF}\cC$). The
goal of this section is to show Proposition~\ref{prop:abs}, a converse of this for finite absorbing
folding problems.

\begin{lemma}\label{lem:reach}
  Let $(G,c,\cC,\cF)$ be a finite folding problem. If $c\preceq_{\cF}\cC$, then $\cC$ is $\cF$-reachable
  from $c$ in at most $\lvert\im(c)\rvert^{e(G)}-1$ steps in the sense that there exist $c'\in\cC$,
  $t\in\NN$ with $t\leq\lvert\im(c)\rvert^{e(G)}-1$ and folds $(f_1,L_1),\ldots,(f_t,L_t)$ such that
  for every $i\in[t]$, either $(f_i,L_i)$ or its dual $(f_i,f_i(L_i))$ is in $\cF$ and
  \begin{equation}\label{eq:reach}
    c' = c\comp (f_1)_{L_1}\comp\cdots\comp (f_t)_{L_t}.
  \end{equation}

  In particular, there exists $\eta\in [2]^t$ with $t\leq\lvert\im(c)\rvert^{e(G)}-1$
  such that for
  \begin{equation*}
    T
    \df
    \{\tau\in[2]^{<\omega} \mid \tau\prec\eta\}
    =
    \{\eta\rest_{[i]} \mid i\in\{0,\ldots,t-1\}\},
  \end{equation*}
  there exists a folding tree $\phi\colon T\to\Fold(G)$ of $G$ such that $\im(\phi)\subseteq\cF$,
  $\phi[c](\eta)\in\cC$ and $\phi[c]^{-1}(\cC)\subseteq L(T\up)$.
\end{lemma}

\begin{proof}
  The second assertion follows from the first by taking a minimum $t$ such that~\eqref{eq:reach}
  holds, defining $\eta\in [2]^t$ and $\phi\colon T\to\Fold(G)$ by
  \begin{align*}
    \eta_i & \df 
    \begin{dcases*}
      1, & if $(f_i,L_i)\in\cF$,\\
      2, & otherwise,
    \end{dcases*}
    &
    \phi(\eta\rest_{[i]}) & \df (f_{i+1},L_{i+1})
  \end{align*}
  and noting that $\phi[c](\eta)=c'$ and the minimality of $t$ ensures that no internal node of
  $T\up$ is labeled by a coloring in $\cC$ (i.e., $\phi[c]^{-1}(\cC)\subseteq L(T\up)$).

  \medskip
    
  Let us prove the first assertion. Since $c\preceq_{\cF}\cC$, we know that such
  $(c',t,(f_i,L_i)_{i\in[t]})$ exist except that the bound $t\leq\lvert\im(c)\rvert^{e(G)}-1$ might
  not hold. Take one such $(c',t,(f_i,L_i)_{i\in[t]})$ that minimizes $t\in\NN$ and for each
  $i\in\{0,\ldots,t\}$, let
  \begin{equation*}
    c_i = c\comp (f_1)_{L_1}\comp\cdots\comp (f_t)_{L_t}.
  \end{equation*}
  Clearly $\im(c_i)\subseteq\im(c)$. On the other hand, the minimality of $t$ implies that the $c_i$
  are pairwise distinct. Since the set
  \begin{equation*}
    \{\widetilde{c}\in C_c^{E(G)} \mid \im(\widetilde{c})\subseteq\im(c)\}
  \end{equation*}
  has size $\lvert\im(c)\rvert^{e(G)}$ (which is finite as $G$ is finite) and contains all $c_i$, it
  follows that $t\leq\lvert\im(c)\rvert^{e(G)}-1$.
\end{proof}

\begin{proposition}\label{prop:abs}
  If $(G,c,\cC,\cF)$ is a finite absorbing folding problem with $c\preceq_{\cF}\cC$, then
  $c\ll_{\cF}\cC$.
\end{proposition}

Before we start the proof, we point out that the main idea of this proof was discovered by
Hatami~\cite[Theorem~2.9(iii)]{Hat10} when $G$ was a hypercube and $\cC$ was the set of
monochromatic colorings, it was then extended to arbitrary finite $G$ and $\cC$ set of monochromatic
colorings by Conlon--Lee~\cite[Theorem~3.2]{CL17}.

\begin{proof}
  Let
  \begin{align*}
    \cR & \df \{r\in C^{E(G)} \mid c\preceq_{\cF} r\land r\preceq_{\cF} \cC\},
    \\
    \cR' & \df
    \{r\comp f_L \mid r\in\cR\setminus\cC\land (f,L)\in\cF\}\cup
    \{r\comp f_L^* \mid r\in\cR\setminus\cC\land (f,L)\in\cF\}.
  \end{align*}
  Note that $c\in\cR$ and since $(G,c,\cC,\cF)$ is absorbing, for every $r\in\cR'$, we have
  $r\preceq_{\cF}\cC$, hence $\cR'\subseteq\cR$.

  Let $B\df\lvert\im(c)\rvert^{e(G)}-1\in\NN$ and note that since every $r\in\cR$ satisfies
  $c\preceq_{\cF} r$, we have $\im(r)\subseteq\im(c)$. By Lemma~\ref{lem:reach}, for each
  $r\in\cR$, there exists $\eta^r\in [2]^{t_r}$ with $t_r\leq B$ such that for
  \begin{equation*}
    T^r
    \df
    \{\tau\in[2]^{<\omega} \mid \tau\prec\eta^r\}
    =
    \{\eta^r\rest_{[i]} \mid i\in\{0,\ldots,t_r-1\}\},
  \end{equation*}
  there exists a folding tree $\phi^r\colon T^r\to\Fold(G)$ of $G$ such that
  $\im(\phi^r)\subseteq\cF$, $\phi^r[r](\eta)\in\cC$ and $\phi^r[r]^{-1}(\cC)\subseteq
  L(T^r\up)$. Note that if $r\in\cC$, then we must have $t_r=0$, so we get
  \begin{equation}\label{eq:muphirrcC}
    \mu_{\phi^r[r]}(\cC) \geq 2^{-t_r} \geq
    \begin{dcases*}
      2^{-B}, & if $r\notin\cC$,\\
      1, & if $r\in\cC$.
    \end{dcases*}
  \end{equation}

  From the structure of $T^r$, we know that every internal node $\tau\in
  T^r\up\setminus L(T^r\up)$ of $T^r\up$ is a prefix of $\eta$ (i.e., $\tau\preceq\eta$), which in
  particular implies that its coloring $\phi^r[r](\tau)$ satisfies
  \begin{equation*}
    c\preceq_{\cF} r\preceq_{\cF} \phi^r[r](\tau) \preceq_{\cF} \phi^r[r](\eta)\in\cC,
  \end{equation*}
  hence $\phi^r[r](\tau)\in\cR$. Since no internal node of $T^r\up$
  is labeled by a coloring in $\cC$, it in turn follows that every leaf $\tau\in L(T^r\up)$ of
  $T^r\up$ satisfies $\phi^r[r](\tau)\in\cR'$, hence is also an element of $\cR$. Thus, we conclude
  that $\im(\phi^r[r])\subseteq\cR$.

  Let us now construct a sequence of finite folding trees $(\phi_n)_{n\in\NN}$ of $G$ with
  $\phi_n\colon T_n\to\Fold(G)$ satisfying the following properties:
  \begin{enumerate}
  \item\label{it:imcF} For every $n\in\NN$, we have $\im(\phi_n)\subseteq\cF$.
  \item\label{it:imcR} For every $n\in\NN$, we have $\im(\phi_n[c])\subseteq\cR$.
  \item\label{it:extends} For every $n\in\NN_+$, we have $T_{n-1}\subseteq T_n$ and
    $\phi_n\rest_{T_{n-1}}=\phi_{n-1}$.
  \item\label{it:measure} For every $n\in\NN_+$, we have
    \begin{equation*}
      1 - \mu_{\phi_n[c]}(\cC) \leq (1 - \mu_{\phi_{n-1}[c]}(\cC))\cdot (1 - 2^{-B}).
    \end{equation*}
  \end{enumerate}
  
  We start by letting $\phi_0\colon\varnothing\to\Fold(G)$ be the empty tree ($T_0=\varnothing$). It
  is clear that $\phi_0$ satisfies items~\ref{it:imcF} and~\ref{it:imcR} (the other items are
  vacuous for $n=0$). Given $\phi_{n-1}\colon T_{n-1}\to\Fold(G)$ satisfying item~\ref{it:imcR},
  since $\im(\phi_n[c])\subseteq\cR$, we can let $\phi_n\colon T_n\to\Fold(G)$ be obtained from
  $\phi_{n-1}$ by appending to each leaf $\tau\in L(T_{n-1}\up)$ the folding tree
  $\phi^{\phi_{n-1}[c](\tau)}$ (note that $\phi_{n-1}[c](\tau)\in\cR$). Formally, we let
  \begin{equation}
    T_n
    \df
    T_{n-1}\cup
    \bigcup_{\tau\in L(T_{n-1}\up)}\{\tau\cdot\tau'\mid \tau'\in T^{\phi_{n-1}[c](\tau)}\},
  \end{equation}
  and note that the union above is of pairwise disjoint sets, so we can let
  \begin{equation*}
    \phi_n(\eta)
    =
    \begin{dcases*}
      \phi_{n-1}(\eta),
      & if $\eta\in T_{n-1}$,
      \\
      \phi^{\phi_{n-1}[c](\tau)}(\tau'),
      & if $\eta=\tau\cdot\tau'$ with $\tau\in L(T_{n-1}\up)$ and $\tau'\in
      T^{\phi_{n-1}[c](\tau)}$.
    \end{dcases*}
  \end{equation*}

  Since for every $r\in\cR$, we have $\im(\phi^r)\subseteq\cF$, it is clear that item~\ref{it:imcF}
  ($\im(\phi_n)\subseteq\cF$) is satisfied. It is also clear that
  \begin{equation*}
    \im(\phi_n[c])
    =
    \im(\phi_{n-1}[c])\cup\bigcup_{\tau\in L(T_{n-1}\up)}\im(\phi^{\phi_{n-1}[c](\tau)}[\phi_{n-1}[c](\tau)])
  \end{equation*}
  and since for every $r\in\cR$, we have $\im(\phi^r[r])\subseteq\cR$, item~\ref{it:imcR}
  ($\im(\phi_n[c])\subseteq\cR$) also follows.

  By construction, it is clear that item~\ref{it:extends} ($T_{n-1}\subseteq T_n$ and
  $\phi_n\rest_{T_{n-1}}=\phi_{n-1}$) holds.

  Finally, since $T_n\supseteq T_{n-1}$, we have
  \begin{align*}
    \mu_{\phi_n[c]}(\cC)
    & =
    \sum_{\tau\in L(T_{n-1}\up)} 2^{-\lvert\tau\rvert}\cdot
    \mu_{\phi^{\phi_{n-1}[c](\tau)}[\phi_{n-1}[c](\tau)]}(\cC)
    \\
    & \geq
    \sum_{\substack{\tau\in L(T_{n-1}\up)\\\phi_{n-1}[c](\tau)\in\cC}} 2^{-\lvert\tau\rvert}
    +
    \sum_{\substack{\tau\in L(T_{n-1}\up)\\\phi_{n-1}[c](\tau)\notin\cC}} 2^{-\lvert\tau\rvert}
    \cdot 2^{-B}
    \\
    & =
    \mu_{\phi_{n-1}[c]}(\cC) + (1-\mu_{\phi_{n-1}[c]}(\cC))\cdot 2^{-B}
    \\
    & =
    1 - (1 - \mu_{\phi_{n-1}[c]}(\cC))\cdot (1 - 2^{-B}),
  \end{align*}
  where the inequality follows from~\eqref{eq:muphirrcC}. Thus, item~\ref{it:measure} also
  holds. This concludes the inductive construction.

  Finally, note that item~\ref{it:measure} implies that
  \begin{equation*}
    \lim_{n\to\infty} \mu_{\phi_n[c]}(\cC)
    \geq
    1 - \lim_{n\to\infty} (1-\mu_{\phi_0[c]}(\cC))\cdot(1-2^{-B})^n
    =
    1,
  \end{equation*}
  which along with items~\ref{it:imcF} and~\ref{it:extends} implies $c\ll_{\cF}\cC$.
\end{proof}

\section{Inequalities from percolation}
\label{sec:ineqsperc}

This section collects inequalities that can be derived from folds and combinatorial properties that
can be derived when such inequalities hold with equality instead.

\begin{lemma}\label{lem:CS}
  Let $H=(G,c)$ be a finite colored bigraph, let $(f,L)\in\Fold(G)$ be a fold of $G$, let $H'\df
  H\rest_{\Fix(f)}$, let $\theta\colon[\lvert\Fix(f)\rvert]\to V(G)$ be an injection with
  $\im(\theta)=\Fix(f)$ and let
  \begin{align*}
    F & \df (H',\theta), &
    F_1 & \df (H\rest_{L\cup\Fix(f)},\theta) - E(H'), &
    F_2 & \df (H\rest_{f(L)\cup\Fix(f)},\theta) - E(H'),
  \end{align*}

  Then for every sequence $W = (W_i)_{i\in C_c}$ of bigraphons
  $W_i\colon\Omega\times\Lambda\to\RR_+$, we have
  \begin{equation}\label{eq:CS}
    t(H,W) \leq \sqrt{t((G,c\comp f_L),W)\cdot t((G,c\comp f_L^*),W)},
  \end{equation}
  with equality if and only if there exist $\lambda_1,\lambda_2\in\RR_+$ not both zero such that
  \begin{equation*}
    \lambda_1\cdot t(F_1,W)(x,y)
    =
    \lambda_2\cdot t(F_2,W)(x,y)
  \end{equation*}
  for $\mu\otimes\nu$-almost every $(x,y)$ such that $t(F,W)(x,y) > 0$.

  Furthermore, \eqref{eq:CS} holds with equality for every sequence $W$ if and only if
  $F_1\rest_{C(F_1)}\cong F_2\rest_{C(F_2)}$.
\end{lemma}

Before we start the proof, we point out that all statements of this lemma except for the final part
(of equality for every sequence $W$) are standard facts from (bi)graphon theory.

\begin{proof}
  For the inequality~\eqref{eq:CS}, note that
  \begin{align*}
    t(H,W)
    & =
    \int_{X^{T_1(F_1)}\times Y^{T_2(F_1)}} t(F_1,W)(x,y)\cdot t(F_2,W)(x)\cdot t(F,W)(x,y)
    \ d(\mu\otimes\nu)(x,y)
    \\
    & \leq
    \begin{multlined}[t]
      \Biggl(
      \int_{X^{T_1(F_1)}\times Y^{T_2(F_1)}} t(F_1,W)(x,y)^2\cdot t(F,W)(x,y)
      \ d(\mu\otimes\nu)(x,y)
      \\
      \cdot
      \int_{X^{T_1(F_1)}\times Y^{T_2(F_1)}} t(F_2,W)(x,y)^2\cdot t(F,W)(x,y)
      \ d(\mu\otimes\nu)(x,y)
      \Biggr)^{1/2}
    \end{multlined}
    \\
    & =
    \sqrt{t((G,c\comp f_L),W)\cdot t((G,c\comp f_L^*),W)},
  \end{align*}
  where the inequality is Cauchy--Schwarz Inequality. Furthermore, equality holds if and only if
  there exist $\lambda_1,\lambda_2\in\RR$ not both zero such that
  \begin{equation}\label{eq:ld}
    \lambda_1\cdot t(F_1,W)(x,y)
    =
    \lambda_2\cdot t(F_2,W)(x,y)
  \end{equation}
  for almost every $(x,y)$ with respect to the measure $\pi$ given by $d\pi(x,y) =
  t(F,W)(x,y)\ d(\mu\otimes\nu)(x,y)$, which means that~\eqref{eq:ld} holds for
  $\mu\otimes\nu$-almost every $(x,y)$ such that $t(F,W)(x,y) > 0$. Since densities are
  non-negative, by possibly multiplying the $\lambda_i$ by $-1$, we may assume that they are both
  non-negative.

  \medskip

  Finally, for the statement that equality in~\eqref{eq:CS} for every sequence $W$ is equivalent to
  $F_1\rest_{C(F_1)}\cong F_2\rest_{C(F_2)}$, we first note that by considering strictly positive
  bigraphons and using the Dominated Convergence Theorem, equality in~\eqref{eq:CS} for every
  sequence $W$ is equivalent to the existence of $\lambda_1,\lambda_2\in\RR_+$ (depending on $W$)
  not both zero such that~\eqref{eq:ld} holds for $\mu\otimes\nu$-almost every $(x,y)$ (without the
  condition $t(F,W)(x,y)>0$), which in turn by Proposition~\ref{prop:coreiso} is equivalent to
  $F_1\rest_{C(F_1)}\cong F_2\rest_{C(F_2)}$.
\end{proof}

The next lemma is also a standard fact of (bi)graphon theory.

\begin{lemma}\label{lem:CStree}
  Let $G$ be a finite bigraph, let $c\colon E(G)\to C$ be a coloring of $G$ and let $\phi\colon
  T\to\Fold(G)$ be a finite folding tree of $G$. Then for every sequence $W=(W_i)_{i\in C}$ of
  bigraphons $W_i\colon\Omega\times\Lambda\to\RR_+$, we have
  \begin{equation*}
    t((G,c),W) \leq t(\phi[c],W).
  \end{equation*}
\end{lemma}

\begin{proof}
  The proof is by induction on the height $h(T)$ of $T$.

  If $T$ is empty (i.e., $h(T)=-\infty$), then $T\up = \{\varnothing\}$, $\phi[c](\varnothing)=c$
  and $\mu_{\phi[c]}$ is the Dirac probability measure concentrated on $c$, so we have $t(\phi[c],W)
  = t((G,c),W)$.

  If $T=\{\varnothing\}$ (i.e., $h(T)=0$), then $T\up = \{1,2\}$ and
  \begin{align*}
    \phi[c](1) & = c\comp\phi_\varnothing, &
    \phi[c](2) & = c\comp\phi_\varnothing^*,
  \end{align*}
  for the folding maps $\phi_\varnothing$ and $\phi_\varnothing^*$ corresponding to the fold
  $\phi(\varnothing)$. Thus, we have
  \begin{equation*}
    t(\phi[c],W) = \sqrt{t((G,c\comp\phi_\varnothing),W)\cdot t((G,c\comp\phi_\varnothing^*),W)}
    \geq
    t((G,c),W),
  \end{equation*}
  where the inequality follows from Lemma~\ref{lem:CS}.

  If $h(T) > 0$, then we let $T'\df T\rest_{h(T)-1}$ and $\psi\df\phi\rest_{h(T)-1}$ and note that
  \begin{align*}
    t((G,c),W)
    & \leq
    t(\psi[c],W)
    =
    \sum_{\eta\in L(T'\up)} t(\psi[c](\eta),W)^{-\lvert\eta\rvert}
    \\
    & =
    \sum_{\eta\in L(T'\up)\cap L(T\up)} t(\phi[c](\eta),W)^{-\lvert\eta\rvert}
    + \sum_{\eta\in L(T'\up)\setminus L(T\up)} t(\phi[c](\eta),W)^{-\lvert\eta\rvert}
    \\
    & \leq
    \begin{multlined}[t]
      \sum_{\eta\in L(T'\up)\cap L(T\up)} t(\phi[c](\eta),W)^{-\lvert\eta\rvert}
      \\
      +
      \sum_{\eta\in L(T'\up)\setminus L(T\up)}
      \left(
      t(\phi[c](\eta)\comp\phi_\eta,W)\cdot t(\phi[c](\eta)\comp\phi_\eta^*,W)
      \right)^{-\lvert\eta\rvert/2}
    \end{multlined}
    \\
    & =
    \sum_{\eta\in L(T\up)} t(\phi[c](\eta),W)^{-\lvert\eta\rvert}
    \\
    & =
    t(\phi[c],W),
  \end{align*}
  where the first inequality follows by induction, the second inequality follows from
  Lemma~\ref{lem:CS} and the second to last equality follows since
  \begin{equation*}
    L(T\up)\setminus L(T'\up)
    =
    \{\eta\cdot i \mid \eta\in L(T'\up)\setminus L(T\up)\land i\in [2]\}.
    \qedhere
  \end{equation*}
\end{proof}

The next lemma is a small generalization of~\cite[Theorem~3.2]{CL17} that covers $c\ll_\cF\cC$ when
$\cC$ is not necessarily the set of monochromatic colorings.

\begin{lemma}\label{lem:CSll}
  If $(G,c,\cC,\cF)$ is a finite folding problem with $c\ll_{\cF}\cC$, then there exists a
  function $\alpha\colon\cC\to\RR_+$ with $\sum_{c'\in\cC} \alpha(c') = 1$
  such that for every sequence $W=(W_i)_{i\in C}$ of bigraphons
  $W_i\colon\Omega\times\Lambda\to\RR_+$, we have
  \begin{equation*}
    t((G,c),W) \leq \prod_{c'\in\cC} t((G,c'),W)^{\alpha(c')}.
  \end{equation*}
\end{lemma}

\begin{proof}
  Let $(\phi_n)_{n\in\NN}$ be a sequence of folding trees of $G$ witnessing $c\ll_{\cF}\cC$. Note
  that the measures $\mu_{\phi_n[c]}$ are supported on the set
  \begin{equation*}
    \cR \df \{r\in C_c^{E(G)} \mid \im(r)\subseteq\im(c)\},
  \end{equation*}
  which is finite since $G$ is finite. Since the set of probability measures supported on $\cR$ is
  compact (in the weak convergence of measures topology), there exists a convergent subsequence
  $(\mu_{\phi_{n_k}[c]})_{k\in\NN}$ that converges to some measure $\nu$ also supported on
  $\cR$. Furthermore, since $\lim_{n\to\infty} \mu_{\phi_n[c]}(\cC)=1$, it follows that
  $\nu(\cC)=1$.

  For each $c'\in\cC$, let $\alpha(c')\df\nu(\{c'\})$ and note that
  $\sum_{c'\in\cC}\alpha(c')=\nu(\cC)=1$. By Lemma~\ref{lem:CStree}, for each $k\in\NN$, we have
  \begin{equation*}
    t((G,c),W) \leq t(\phi_{n_k}[c],W) = \prod_{r\in\cR} t((G,r),W)^{\mu_{\phi_{n_k}[c]}(\{r\})}
  \end{equation*}
  and since the right-hand side (is a finite product and) converges to
  \begin{equation*}
    \prod_{c'\in\cC\cap\cR} t((G,c'),W)^{\alpha(c')}\cdot\prod_{r\in\cR\setminus\cC} t((G,r),W)^0
    =
    \prod_{c'\in\cC} t((G,c'),W)^{\alpha(c')},
  \end{equation*}
  the result follows.
\end{proof}

\section{Fold-stability}
\label{sec:foldstability}

In this section, we establish all implications (under appropriate hypotheses) between the several
variants of fold-stability and $\preceq_\cF$-maximality.

\begin{lemma}\label{lem:foldstable}
  Let $G$ be a bigraph, let $\cF,\cF'\subseteq\Fold(G)$ be sets of folds of $G$ and let $K$ be a
  subgroup of the automorphism group $\Aut(G)$ of $G$. Then the following hold for a coloring
  $c\colon E(G)\to C$ of $G$:
  \begin{enumerate}
  \item\label{lem:foldstable:strong} If $(h,L),(f,L)\in\Fold(G)$ are folds of $G$ with the same $L$
    and $\Fix(h)=\Fix(f)$ and $h\in\Aut((G,c))$, then $c\comp f_L\comp h_L = c$ and
    $(G,c)\cong(G,c\comp f_L)$.
  \item\label{lem:foldstable:reachability} If $\cF$ is $K$-invariant and closed under dual folds and
    $c$ is $(K,\cF)$-fold-stable, then for every $c'\colon E(G)\to C$ with $c\preceq_{\cF} c'$, $c'$ is
    $(K,\cF)$-fold-stable and $(G,c')$ is isomorphic to $(G,c)$ under a bijection in $K$.
  \item\label{lem:foldstable:reachabilitystrong} If $\cF$ and $\cF'$ are $K$-invariant, $\cF$ is
    closed under dual folds and $c$ is $\cF'$-strongly $(K,\cF)$-fold-stable, then for every
    $c'\colon E(G)\to C$ with $c\preceq_{\cF} c'$, $c'$ is $\cF'$-strongly $(K,\cF)$-fold-stable.
  \item\label{lem:foldstable:maximal} If $\cF$ is $K$-invariant and closed under dual folds and $c$
    is $\cF$-strongly $(K,\cF)$-fold-stable, then $c$ is $\preceq_{\cF}$-maximal.
  \item\label{lem:foldstable:CSequal->strong} If $G$ is finite and connected and $(f,L)\in\Fold(G)$
    is a fold such that for every sequence $W = (W_i)_{i\in C}$ of bigraphons
    $W_i\colon\Omega\times\Lambda\to\RR_+$, we have
    \begin{equation}\label{eq:CSequal}
      t((G,c),W) = \sqrt{t((G,c\comp f_L),W)\cdot t((G,c\comp f_L^*),W)},
    \end{equation}
    then $c$ is strongly $\{(f,L)\}$-fold-stable.
  \item\label{lem:foldstable:weak->strong} If $G$ is finite and connected, $\cF$ is closed under
    dual folds and $c$ is $(K,\cF)$-fold-stable, then $c$ is strongly $(K,\cF)$-fold-stable.
  \item\label{lem:foldstable:symmfoldstable} If there exists a rainbow coloring $r\colon E(G)\to C$
    such that $r\preceq_{\cF} c$, $\cF$ is $K$-invariant and closed under dual folds, $G$ is
    $K'$-edge-transitive, where $K'$ is the subgroup of $\Aut(G)$ generated by $\{f \mid
    (f,L)\in\cF\}$, and $c$ is $\cF'$-strongly $(K,\cF)$-fold-stable, then $c$ is
    $\cF'$-symmetrically $(K,\cF)$-fold-stable.
  \end{enumerate}
\end{lemma}

\begin{proof}
  For item~\ref{lem:foldstable:strong}, since every edge of $G$ is either contained in
  $L\cup\Fix(h)=L\cup\Fix(f)$ or contained in $h(L)\cup\Fix(h)=f(L)\cup\Fix(f)$, we have
  \begin{align*}
    (c\comp f_L\comp h_L)(u,v)
    & =
    \begin{dcases*}
      (c\comp f_L)(u,v), & if $u,v\in L\cup\Fix(h)$,\\
      (c\comp f_L\comp h)(u,v), & if $u,v\in h(L)\cup\Fix(h)$,
    \end{dcases*}
    \\
    & =
    \begin{dcases*}
      c(u,v), & if $u,v\in L\cup\Fix(h)$,\\
      c(h(u),h(v)), & if $u,v\in h(L)\cup\Fix(h)$,
    \end{dcases*}
    \\
    & =
    c(u,v),
  \end{align*}
  where the second equality follows since $h$ is an involution and the third equality follows since
  $h$ preserves the coloring $c$.

  We now show that $(G,c)\cong (G,c\comp f_L)$. Define the function $g\colon V(G)\to V(G)$ by
  \begin{equation*}
    g(v) \df
    \begin{dcases*}
      v, & if $v\in L\cup\Fix(f)$,\\
      (f\comp h)(v), & if $v\in f(L)$.
    \end{dcases*}
  \end{equation*}
  Note that since $\Fix(f)=\Fix(h)$, we also have $g(v) = (f\comp h)(v)$ for $v\in\Fix(f)$.

  We claim that $g$ is an isomorphism from $(G,c)$ to $(G,c\comp f_L)$. Since both $f$ and $h$ are
  bijective and $(f\comp h\comp f)(L)=f(L)$, it is clear that $g$ is bijective.

  To see that $g$ preserves edges, we simply note that every edge of $G$ is either contained in
  $L\cup\Fix(f)$, in which case it is preserved as $g$ acts identically on this set; or is contained
  in $f(L)\cup\Fix(f)$, in which case it is preserved as $g$ acts as $f\comp h$ on this set.

  To see that $g$ also preserves non-edges, a similar argument shows that $g$ preserves any non-edge
  that is contained in $L\cup\Fix(f)$ or contained in $f(L)\cup\Fix(f)$; since $\Fix(f)$ is a cut of
  $G$, all remaining non-edges amount to \emph{all} elements of
  \begin{equation*}
    \Bigl((V_1(G)\cap L)\times (V_2(G)\cap f(L))\Bigr)
    \cup
    \Bigl((V_1(G)\cap f(L))\times (V_2(G)\cap L)\Bigr).
  \end{equation*}
  Since $(f\comp h\comp f)(L)=f(L)$, it follows that the set above is fixed (as a set) by $g$, hence
  $g$ also preserves non-edges.

  Let us now show that $g$ respects colorings. Note that if $(u,v)\in E(G)$, then either $u,v\in
  L\cup\Fix(f)$ or $u,v\in f(L)\cup\Fix(f)$, so we get
  \begin{align*}
    (c\comp f_L\comp g)(u,v)
    & =
    \begin{dcases*}
      (c\comp f_L)(u,v), & if $u,v\in L\cup\Fix(f)$,\\
      (c\comp f_L\comp f\comp h)(u,v), & if $u,v\in f(L)\cup\Fix(f)$,
    \end{dcases*}
    \\
    & =
    \begin{dcases*}
      c(u,v), & if $u,v\in L\cup\Fix(f)$,\\
      (c\comp f\comp f\comp h)(u,v), & if $u,v\in f(L)\cup\Fix(f)$,
    \end{dcases*}
    \\
    & =
    c(u,v),
  \end{align*}
  where the second equality follows since $(f\comp h)(f(L)\cup\Fix(f))=f(L)\cup\Fix(f)$ and $f_L$
  acts as $f$ on this set and the third equality follows since $f$ is an involution and $h$ is
  preserves $c$.

  \medskip

  For item~\ref{lem:foldstable:reachability}, first note that since $\cF$ is closed under dual
  folds, every $c'\colon E(G)\to C$ with $c\preceq_{\cF} c'$ can be written as
  \begin{equation}\label{eq:c'}
    c' = c\comp (f_1)_{L_1}\comp\cdots\comp (f_n)_{L_n}
  \end{equation}
  for some sequence of folds $(f_1,L_1),\ldots,(f_n,L_n)\in\cF$. Let us show by induction in $n$
  that such $c'$ is $(K,\cF)$-fold-stable and $(G,c)$ and $(G,c')$ are isomorphic under a bijection
  in $K$. For $n=0$, this is trivial. For $n > 0$, by inductive hypothesis, we know that
  \begin{equation}\label{eq:c''}
    c'' \df c\comp (f_1)_{L_1}\comp\cdots\comp (f_{n-1})_{L_{n-1}}
  \end{equation}
  is $(K,\cF)$-fold-stable and satisfies $(G,c)\cong (G,c'')$ under a bijection in $K$. Since
  $c'=c''\comp (f_n)_{L_n}$, by $(K,\cF)$-fold-stability of $c''$, we get $(G,c')\cong (G,c'')$
  under a bijection in $K$, so we also get $(G,c')\cong (G,c)$ under a bijection in $K$. Let then
  $h\in K$ be an isomorphism from $(G,c')$ to $(G,c)$ and note that if $(f,L)\in\cF$, then
  \begin{equation*}
    c'\comp f_L
    =
    (c\comp h)\comp f_L
    =
    c\comp(h\comp f_L\comp h^{-1})\comp h
    =
    c\comp(h\cdot f\comp h^{-1})_{h(L)}\comp h,
  \end{equation*}
  where the last equality follows from Lemma~\ref{lem:action}. Hence $h$ is an isomorphism from
  $(G,c'\comp f_L)$ to $(G,c\comp f'_{L'})$ for the fold $(f',L')\df h\cdot (f,L)$. Since $\cF$ is
  $K$-invariant, we have $f'\in\cF$, so $(K,\cF)$-fold-stability of $c$ yields $(G,c\comp
  f'_{L'})\cong (G,c)$ under a bijection in $K$, from which we conclude that $(G,c'\comp f_L)\cong
  (G,c')$ under a bijection in $K$, so $c'$ is $(K,\cF)$-fold-stable.

  \medskip

  For item~\ref{lem:foldstable:reachabilitystrong}, clearly $(K,\cF)$-fold-stability follows from
  item~\ref{lem:foldstable:reachability}. Fix then $(f',L')\in\cF$ and let us show that there exists
  $(h',L')\in\cF'$ with $\Fix(h')=\Fix(f')$ and $h\in\Aut((G,c'))$.

  By item~\ref{lem:foldstable:reachability}, there exists $g\in K$ that is an isomorphism from
  $(G,c)$ to $(G,c')$ and since $\cF$ is $K$-invariant, the fold
  \begin{equation*}
    (f,L)
    \df
    g^{-1}\cdot(f',L')
    =
    (g^{-1}\comp f'\comp g, g^{-1}(L'))
  \end{equation*}
  is in $\cF$, so $\cF'$-strong $(K,\cF)$-fold-stability of $c$ yields an $h\in\Aut((G,c))$ such
  that $(h,L)\in\cF'$ and $\Fix(h)=\Fix(f)$. Let $h'\df g\comp h\comp g^{-1}$. Clearly $h'$ is an
  automorphism of $G$. The fact that $h'$ preserves $c'$ follows since $g$ is an isomorphism from
  $(G,c)$ to $(G,c')$ and $h$ preserves $c$. Since by Lemma~\ref{lem:action}
  \begin{equation*}
    g\cdot (h,L) = (g\comp h\comp g^{-1}, g(L)) = (h',L')
  \end{equation*}
  is a fold of $G$ with
  \begin{equation*}
    \Fix(h') = \Fix(g\comp h\comp g^{-1}) = g(\Fix(h)) = g(\Fix(f))
    = \Fix(g\comp f\comp g^{-1}) = g(\Fix(f)) = \Fix(f')
  \end{equation*}
  and since $g\in K$ and $\cF'$ is $K$-invariant, it follows that $(h',L')\in\cF'$, so $c'$ is
  $\cF'$-strongly $(K,\cF)$-fold-stable.

  \medskip

  For item~\ref{lem:foldstable:maximal}, let
  \begin{equation*}
    \cR \df \{c'\colon E(G)\to C \mid c\preceq_{\cF} c'\}.
  \end{equation*}
  By item~\ref{lem:foldstable:reachabilitystrong}, we know that every $c'\in\cR$ is $\cF$-strongly
  $(K,\cF)$-fold-stable. Thus, to show that $c$ is $\preceq_{\cF}$-maximal, (since $\cF$ is closed
  under dual folds) it suffices to show that for every $c'\in\cR$ and every $(f,L)\in\cF$, we have
  $c'\comp f_L\preceq_{\cF} c'$. Since $c'$ is $\cF$-strongly $(K,\cF)$-fold-stable, there exists
  $h\in\Aut((G,c'))$ such that $(h,L)\in\cF$ and $\Fix(h)=\Fix(f)$. By
  item~\ref{lem:foldstable:strong}, we have
  \begin{equation*}
    c'\comp f_L\comp h_L = c',
  \end{equation*}
  that is, we have $c'\comp f_L\preceq_{\cF} c'$, as desired.

  \medskip

  We now prove item~\ref{lem:foldstable:CSequal->strong}. By Lemma~\ref{lem:CS} along with the
  equality~\eqref{eq:CSequal}, we have $F_1\rest_{C(F_1)}\cong F_2\rest_{C(F_2)}$ for the flags
  \begin{align*}
    F_1 & \df (H\rest_{L\cup\Fix(f)},\theta) - E(H'), &
    F_2 & \df (H\rest_{f(L)\cup\Fix(f)},\theta) - E(H'),
  \end{align*}
  where $H\df(G,c)$, $H'\df H\rest_{\Fix(f)}$ and $\theta\colon[\lvert\Fix(f)\rvert]\to V(G)$ is any
  injection with $\im(\theta)=\Fix(f)$.

  Since all edges of $H'$ are contained in $\Fix(f)$, it follows that for the flags
  \begin{align*}
    F'_1 & \df (H\rest_{L\cup\Fix(f)},\theta), &
    F'_2 & \df (H\rest_{f(L)\cup\Fix(f)},\theta),
  \end{align*}
  we also have $F'_1\rest_{C(F'_1)}\cong F'_2\rest_{C(F'_2)}$. Since $G$ is connected, we have
  $C(F'_i)=V(F'_i)$ ($i\in[2]$), so we get $F'_1\cong F'_2$.

  Let then $g$ be an isomorphism from $F'_1$ to $F'_2$. Since $g$ preserves the flag labeling, we
  must have $\Fix(g)\supseteq\im(\theta)=\Fix(f)$ and in fact $\Fix(g)=\Fix(f)$ as $V(F'_1)\cap
  V(F'_2)=\Fix(f)$.

  Let $h\colon V(G)\to V(G)$ act as $g$ on $L\cup\Fix(f))$ and as $g^{-1}$ on $f(L)$. Since
  \begin{align*}
    \Fix(g) & = \Fix(f), &
    H(F'_1) & = H\rest_{L\cup\Fix(f)}, &
    H(F'_2) & = H\rest_{f(L)\cup\Fix(f)},
  \end{align*}
  and $\Fix(g)$ is a cut of $G$, it follows that $h$ is an automorphism of $H=(G,c)$ such that
  $\Fix(h)=\Fix(f)$ and $(h,L)$ is a fold of $G$. Thus $c$ is strongly $\{(f,L)\}$-fold-stable.

  \medskip

  For item~\ref{lem:foldstable:weak->strong}, let $(f,L)\in\cF$ and since $G$ is finite, by
  Lemma~\ref{lem:CS}, for every sequence $W = (W_i)_{i\in C_c}$ of bigraphons
  $W_i\colon\Omega\times\Lambda\to\RR_+$, we have
  \begin{equation*}
    t((G,c),W) \leq \sqrt{t((G,c\comp f_L),W)\cdot t((G,c\comp f_L^*),W)}
  \end{equation*}
  and since $c$ is $(K,\cF)$-fold-stable and $\cF$ is closed under dual folds, the above in fact
  holds with equality as $(G,c\comp f_L)\cong (G,c)\cong (G,c\comp f_L^*)$. Thus, by
  item~\ref{lem:foldstable:CSequal->strong}, we get that $c$ is strongly $(K,\{(f,L)\})$-fold-stable
  and since $(f,L)$ was chosen arbitrarily in $\cF$, it follows that $c$ is strongly
  $(K,\cF)$-fold-stable.

  \medskip

  For item~\ref{lem:foldstable:symmfoldstable}, since $\cF$ is closed under dual folds and
  $r\preceq_{\cF} c$, there exists a sequence of folds $(f_1,L_1),\ldots,(f_n,L_n)\in\cF$ such that
  \begin{equation}\label{eq:cr}
    c = r\comp (f_1)_{L_1}\comp\cdots\comp (f_n)_{L_n}.
  \end{equation}

  Let $i,j\in\im(c)$. Since $i\in\im(c)\subseteq\im(r)$ and $r$ is injective, there exists a unique
  edge $e_i\in E(G)$ such that $r(e_i)=i$. On the other hand, since $j\in\im(c)$, there must exist
  at least one edge $e_j\in E(G)$ such that $c(e_j)=j$.

  Since $G$ is $K'$-edge-transitive, there must exist a sequence of folds
  $(f'_1,L'_1),\ldots,(f'_m,L'_m)\in\cF$ such that $(f'_1\comp\cdots\comp f'_m)(e_i)=e_j$ and since
  $\cF$ is closed under dual folds, we may suppose without loss of generality that for every
  $t\in[m]$, the edge
  \begin{equation*}
    (f'_{t+1}\comp\cdots\comp f'_m)(e_i)
  \end{equation*}
  is contained in $f'_t(L'_t)\cup\Fix(f'_t)$. This in particular implies that
  \begin{equation}\label{eq:f'ei}
    ((f'_1)_{L'_1}\comp\cdots\comp(f'_m)_{L'_m})(e_i) = e_j.
  \end{equation}

  Let
  \begin{equation}\label{eq:c'c}
    c' \df c\comp(f'_1)_{L'_1}\comp\cdots\comp(f'_m)_{L'_m}\comp(f_1)_{L_1}\comp\cdots\comp(f_n)_{L_n}.
  \end{equation}
  We claim that $c' = \sigma\comp c$ for some permutation $\sigma\colon C\to C$ with
  $\sigma(i)=j$. To prove this, it suffices to show that:
  \begin{enumerate}[label={\alph*.}, ref={(\alph*)}]
  \item\label{it:c'level} If $e,e'\in E(G)$ are such that $c(e)=c(e')$, then $c'(e)=c'(e')$.
  \item\label{it:c'ij} For every $e\in c^{-1}(i)$, we have $c'(e)=j$.
  \end{enumerate}

  For item~\ref{it:c'level}, note that $c(e)=c(e')$ along with~\eqref{eq:cr} and the fact that $r$
  is injective implies that
  \begin{equation*}
    (f_1)_{L_1}\comp\cdots\comp (f_n)_{L_n}(e) = (f_1)_{L_1}\comp\cdots\comp (f_n)_{L_n}(e'),
  \end{equation*}
  so $c'(e)=c'(e')$ follows from~\eqref{eq:c'c}.

  For item~\ref{it:c'ij}, since $c(e)=i$, \eqref{eq:cr} along with the fact that $r$ is injective
  implies that
  \begin{equation*}
    (f_1)_{L_1}\comp\cdots\comp (f_n)_{L_n}(e) = e_i,
  \end{equation*}
  which along with~\eqref{eq:f'ei} and~\eqref{eq:c'c} yields
  \begin{align*}
    c'(e)
    & =
    (c\comp(f'_1)_{L'_1}\comp\cdots\comp(f'_m)_{L'_m}\comp(f_1)_{L_1}\comp\cdots\comp(f_n)_{L_n})(e)
    \\
    & =
    (c\comp(f'_1)_{L'_1}\comp\cdots\comp(f'_m)_{L'_m})(e_i)
    \\
    & =
    c(e_j)
    =
    j.
  \end{align*}

  Finally, the definition of $c'$ implies that $c\preceq_{\cF} c'$, so
  item~\ref{lem:foldstable:reachability} implies that $(G,c)$ is isomorphic to $(G,c')$ under a
  bijection in $K$.
\end{proof}

The next lemma is the left-sided analogue of
Lemma~\ref{lem:foldstable}\ref{lem:foldstable:symmfoldstable}.

\begin{lemma}\label{lem:leftsymmfoldstable}
  Let $G$ be a bigraph, let $C$ be a set of colors, let $c\colon E(G)\to C'$ be a coloring of $G$,
  let $\cF,\cF'\subseteq\Fold((G,c))$ be sets of folds of the colored bigraph $(G,c)$, let $K$ be a
  subgroup of the automorphism group $\Aut((G,c))$ of $(G,c)$, and let $r\colon V_1(G)\to C$ be a
  rainbow left-coloring of $G$.

  Suppose that $\cF$ is $K$-invariant and closed under dual folds, $G$ is
  $K'$-left-vertex-transitive, where $K'$ is the subgroup of $\Aut((G,c))$ generated by $\{f\mid
  (f,L)\in\cF\}$.

  If $\ell\colon V_1(G)\to C$ is a left-coloring of $G$ such that $r\otimes
  c\preceq_{\cF}\ell\otimes c$ and $\ell\otimes c$ is $\cF'$-strongly $(K,\cF)$-fold-stable, then
  $c$ is $\cF'$-left-symmetrically $(K,\cF)$-fold-stable.
\end{lemma}

\begin{proof}
  The proof is analogous to that of Lemma~\ref{lem:foldstable}\ref{lem:foldstable:symmfoldstable}.

  Since $\cF$ is closed under dual folds and $r\otimes c\preceq_{\cF}\ell\otimes c$, there exists a
  sequence of folds $(f_1,L_1),\ldots,(f_n,L_n)\in\cF$ such that
  \begin{equation}\label{eq:ellr}
    \ell\otimes c
    =
    (r\otimes c)\comp(f_1)_{L_1}\comp\cdots\comp (f_n)_{L_n}
    =
    (r\comp(f_1)_{L_1}\comp\cdots\comp (f_n)_{L_n})\otimes c.
  \end{equation}

  Let $i,j\in\im(\ell)$. Since $i\in\im(\ell)\subseteq\im(r)$ and $r$ is injective, there exists a
  unique vertex $u_i\in V_1(G)$ such that $r(u_i)=i$. On the other hand, since $j\in\im(\ell)$,
  there must exist at least one vertex $u_j\in V_1(G)$ such that $\ell(u_j)=j$.

  Since $G$ is $K'$-left-vertex-transitive, there must exist a sequence of folds
  $(f'_1,L'_1),\ldots,(f'_m,L'_m)\in\cF$ such that $(f'_1\comp\cdots\comp f'_m)(u_i)=u_j$ and since
  $\cF$ is closed under dual folds, we may suppose without loss of generality that for every
  $t\in[m]$, the vertex
  \begin{equation*}
    (f'_{t+1}\comp\cdots\comp f'_m)(u_i)
  \end{equation*}
  is contained in $f'_t(L'_t)\cup\Fix(f'_t)$. This in particular implies that
  \begin{equation}\label{eq:f'ui}
    ((f'_1)_{L'_1}\comp\cdots\comp(f'_m)_{L'_m})(u_i) = u_j.
  \end{equation}

  Let
  \begin{equation}\label{eq:ell'ell}
    \ell' \df \ell\comp(f'_1)_{L'_1}\comp\cdots\comp(f'_m)_{L'_m}\comp(f_1)_{L_1}\comp\cdots\comp(f_n)_{L_n}.
  \end{equation}
  We claim that $\ell' = \sigma\comp\ell$ for some permutation $\sigma\colon C\to C$ with
  $\sigma(i)=j$. To prove this, it suffices to show that:
  \begin{enumerate}[label={\alph*.}, ref={(\alph*)}]
  \item\label{it:ell'level} If $u,u'\in V_1(G)$ are such that $\ell(u)=\ell(u')$, then
    $\ell'(u)=\ell'(u')$.
  \item\label{it:ell'ij} For every $u\in\ell^{-1}(i)$, we have $\ell'(u)=j$.
  \end{enumerate}

  For item~\ref{it:ell'level}, note that $\ell(u)=\ell(u')$ along with~\eqref{eq:ellr} and the fact
  that $r$ is injective implies that
  \begin{equation*}
    (f_1)_{L_1}\comp\cdots\comp (f_n)_{L_n}(u) = (f_1)_{L_1}\comp\cdots\comp (f_n)_{L_n}(u'),
  \end{equation*}
  so $\ell'(u)=\ell'(u')$ follows from~\eqref{eq:ell'ell}.

  For item~\ref{it:ell'ij}, since $\ell(u)=i$, \eqref{eq:ellr} along with the fact that $r$ is
  injective implies that
  \begin{equation*}
    (f_1)_{L_1}\comp\cdots\comp (f_n)_{L_n}(u) = u_i,
  \end{equation*}
  which along with~\eqref{eq:f'ui} and~\eqref{eq:ell'ell} yields
  \begin{align*}
    \ell'(u)
    & =
    (\ell\comp(f'_1)_{L'_1}\comp\cdots\comp(f'_m)_{L'_m}\comp(f_1)_{L_1}\comp\cdots\comp(f_n)_{L_n})(u)
    \\
    & =
    (\ell\comp(f'_1)_{L'_1}\comp\cdots\comp(f'_m)_{L'_m})(u_i)
    \\
    & =
    \ell(u_j)
    =
    j.
  \end{align*}

  Finally, the definition of $\ell'$ implies that $\ell\otimes c\preceq_{\cF}\ell'\otimes c$, so
  Lemma~\ref{lem:foldstable}\ref{lem:foldstable:reachability} implies that $(G,\ell\otimes c)$ is
  isomorphic to $(G,\ell'\otimes c)$ under a bijection in $K$.
\end{proof}

\begin{lemma}\label{lem:maximal}
  Let $G$ be a finite bigraph, let $C$ be a set of colors and let $\cF\subseteq\Fold(G)$ be a set of
  folds of $G$. The following hold for a $\preceq_{\cF}$-strongly connected set $\cC\subseteq
  C^{E(G)}$ of $\preceq_{\cF}$-maximal elements:
  \begin{enumerate}
  \item\label{lem:maximal:absorbing} For every $c\in\cC$ and every non-empty $\cC'\subseteq\cC$, the
    folding problem $(G,c,\cC',\cF)$ is absorbing.
  \item\label{lem:maximal:percolation} For every $c\in\cC$ and every non-empty $\cC'\subseteq\cC$,
    we have $c\ll_{\cF}\cC'$.
  \item\label{lem:maximal:stronglyfoldstable} If $G$ is connected, then for every $c,c'\in\cC$, we
    have $(G,c)\cong (G,c')$ and $c$ is strongly $\cF$-fold-stable.
  \end{enumerate}
\end{lemma}

\begin{proof}
  For item~\ref{lem:maximal:absorbing}, let $c'\in\cC'$, let $r\in C^{E(G)}\setminus\cC'$ be such
  that $c\preceq_{\cF} r$ and let $(f,L)\in\cF$. Since $c$ is $\preceq_{\cF}$-maximal, we must have
  $r\comp f_L\preceq_{\cF} c$ and $r\comp f_L^*\preceq_{\cF} c$. Since $\cC$ is
  $\preceq_{\cF}$-strongly connected and $c,c'\in\cC$, we have $c\preceq_{\cF} c'$, hence $r\comp
  f_L\preceq_{\cF} c'$ and $r\comp f_L^*\preceq_{\cF} c'$, thus $r\comp f_L\preceq_{\cF}\cC'$ and
  $r\comp f_L^*\preceq_{\cF}\cC'$. Therefore $(G,c,\cC',\cF)$ is absorbing.

  \medskip

  For item~\ref{lem:maximal:percolation}, since $\cC$ is $\preceq_{\cF}$-strongly connected and
  $\cC'\subseteq\cC$ is non-empty, we know that $c\preceq_{\cF}\cC'$ and since
  item~\ref{lem:maximal:absorbing} says $(G,c,\cC',\cF)$ is absorbing, by
  Proposition~\ref{prop:abs}, it follows that $c\ll_{\cF}\cC'$.

  \medskip

  For item~\ref{lem:maximal:stronglyfoldstable}, let
  \begin{equation*}
    \cR \df \{r\in C^{E(G)} \mid c\preceq_{\cF} r\}
  \end{equation*}
  be the set of colorings that are $\cF$-reachable from $c$. Since $c\in\cR$ and $c$ is
  $\preceq_{\cF}$-maximal, it follows that $\cR$ is $\preceq_{\cF}$-strongly connected, so by
  item~\ref{lem:maximal:percolation}, for every $r\in\cR$, we have $r\ll_{\cF}\{c\}$, in particular,
  for a given fold $(f,L)\in\cF$, we have $c\comp f_L\ll_{\cF}\{c\}$ and $c\comp
  f_L^*\ll_{\cF}\{c\}$.

  Applying Lemma~\ref{lem:CS} followed by Lemma~\ref{lem:CSll} for $c\comp f_L$ and $c\comp f_L^*$,
  it follows that for every sequence $W = (W_i)_{i\in C}$ of bigraphons
  $W_i\colon\Omega\times\Lambda\to\RR_+$, we have
  \begin{equation*}
    t((G,c),W)
    \leq
    \sqrt{t((G,c\comp f_L),W)\cdot t((G,c\comp f_L^*),W)}
    \leq
    t((G,c),W),
  \end{equation*}
  so we must have equality throughout, hence by Lemma~\ref{lem:foldstable:CSequal->strong}, it
  follows that $c$ is strongly $\{(f,L)\}$-fold-stable. Since $(f,L)$ was chosen arbitrarily in
  $\cF$, it follows that $c$ is strongly $\cF$-fold-stable.

  To see that for every $c,c'\in\cC$, we have $(G,c)\cong (G,c')$, since $\cC$ is
  $\preceq_{\cF}$-strongly connected, it suffices to show that for every fold $(f,L)\in\cF$, we have
  \begin{equation*}
    (G,c)\cong (G,c\comp f_L)\cong (G,c\comp f_L^*).
  \end{equation*}
  But this follows from strong $\cF$-fold-stability along with
  Lemma~\ref{lem:foldstable}\ref{lem:foldstable:strong}.
\end{proof}

The next corollary summarizes Lemmas~\ref{lem:foldstable} and~\ref{lem:maximal} into a sufficient
condition for all versions of fold-stability and maximality to be equivalent.

\begin{corollary}\label{cor:obst}
  Let $G$ be a finite connected bigraph, let $\cF\subseteq\Fold(G)$ be a set of folds of $G$ that is
  $\Aut(G)$-invariant. The following are equivalent for a coloring $c\colon E(G)\to C$ of $G$:
  \begin{enumerate}
  \item\label{cor:obst:foldstable} $c$ is $\cF$-fold-stable.
  \item\label{cor:obst:stronglyfoldstable} $c$ is $\cF$-strongly $\cF$-fold-stable.
  \item\label{cor:obst:maximal} $c$ is $\preceq_{\cF}$-maximal.
  \end{enumerate}

  Furthermore, if there exists a rainbow coloring $r\colon E(G)\to C$ such that $r\preceq_{\cF} c$
  and $G$ is $K'$-edge-transitive, where $K'$ is the subgroup of $\Aut(G)$ generated by $\{f\mid
  (f,L)\in\cF\}$, then the following is also equivalent to the above:
  \begin{enumerate}[resume]
  \item\label{cor:obst:symmfoldstable} $c$ is $\cF$-symmetrically $\cF$-fold-stable.
  \end{enumerate}
\end{corollary}

\begin{proof}
  First note that by Lemma~\ref{lem:dualfold}, $\cF$ is closed under dual folds.
  
  The implications~\ref{cor:obst:stronglyfoldstable}$\implies$\ref{cor:obst:foldstable}
  and~\ref{cor:obst:symmfoldstable}$\implies$\ref{cor:obst:foldstable} are trivial.

  The implication~\ref{cor:obst:foldstable}$\implies$\ref{cor:obst:stronglyfoldstable} is
  Lemma~\ref{lem:foldstable}\ref{lem:foldstable:weak->strong}.

  The implication~\ref{cor:obst:stronglyfoldstable}$\implies$\ref{cor:obst:maximal} is
  Lemma~\ref{lem:foldstable}\ref{lem:foldstable:maximal}.

  The implication~\ref{cor:obst:maximal}$\implies$\ref{cor:obst:stronglyfoldstable} is
  Lemma~\ref{lem:maximal}\ref{lem:maximal:stronglyfoldstable} with
  \begin{equation*}
    \cC \df \{c'\in C^{E(G)} \mid c\preceq_{\cF} c'\}.
  \end{equation*}

  The implication~\ref{cor:obst:foldstable}$\implies$\ref{cor:obst:symmfoldstable} under the
  additional assumptions of $r\preceq_{\cF} c$ and $K'$-edge-transitivity of $G$ is
  Lemma~\ref{lem:foldstable}\ref{lem:foldstable:symmfoldstable}.
\end{proof}

\section{Cut-percolating bigraphs}
\label{sec:cutperc}

In this section, we put together all results proven so far to characterize cut-percolation
(Theorem~\ref{thm:cutperc}) and left-cut-percolation (Theorem~\ref{thm:leftcutperc}) in terms of
non-existence of (any of) the non-monochromatic obstacles.

\begin{theorem}\label{thm:cutperc}
  Let $G$ be a finite connected bigraph, let $C$ be a set of colors, let $\cF\subseteq\Fold(G)$ be a
  set of folds of $G$ that is $\Aut(G)$-invariant, let $\cM\subseteq C^{E(G)}$ be the set of
  monochromatic colors of $G$ by $C$ and let $r\colon E(G)\to C$ be a rainbow coloring of $G$. Then
  the following are equivalent:
  \begin{enumerate}
  \item\label{thm:cutperc:cutperc} $G$ is $\cF$-cut-percolating.
  \item\label{thm:cutperc:reachevery} $c\preceq_{\cF}\cM$ for every $c\colon E(G)\to C$.
  \item\label{thm:cutperc:reach} $r\preceq_{\cF}\cM$.
  \item\label{thm:cutperc:llevery} $c\ll_{\cF}\cM$ for every $c\colon E(G)\to C$.
  \item\label{thm:cutperc:ll} $r\ll_{\cF}\cM$.
  \item\label{thm:cutperc:foldstable} Every $\cF$-fold-stable coloring $c\colon E(G)\to C$ of $G$ is
    monochromatic.
  \item\label{thm:cutperc:stronglyfoldstable} Every strongly $\cF$-fold-stable coloring $c\colon
    E(G)\to C$ of $G$ is monochromatic.
  \item\label{thm:cutperc:maximal} Every $\preceq_{\cF}$-maximal coloring $c\colon E(G)\to C$ of $G$
    is monochromatic.
  \item\label{thm:cutperc:symmfoldstable} Every symmetrically $\cF$-fold-stable coloring $c\colon
    E(G)\to C$ of $G$ is monochromatic and $G$ is $K$-edge-transitive, where $K$ is the subgroup of
    $\Aut(G)$ generated by $\{f\mid (f,L)\in\cF\}$.
  \end{enumerate}
\end{theorem}

Before we prove this theorem, let us note two particular cases to which it can be applied: one is
when $\cF=\Fold(G)$ (which is trivially $\Aut(G)$-invariant), which characterizes when $G$ is
cut-percolating; and the other is when $\cF=\IndFold(G)$ (which is $\Aut(G)$-invariant by
Lemma~\ref{lem:indfold}), which characterizes when $G$ is independently cut-percolating.

\begin{proof}
  The structure of the proof is as follows:
  
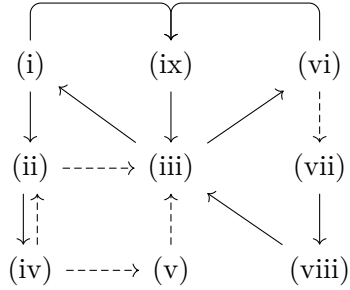
\begin{equation*}
    \begin{tikzcd}
      \text{\ref{thm:cutperc:cutperc}}
      \arrow[d]
      \arrow[r, rounded corners,%
        to path={%
          -- ([yshift={3ex}]\tikztostart.north)%
          -- ([yshift={3ex}]\tikztotarget.north)%
          -- (\tikztotarget)}]
      &
      \text{\ref{thm:cutperc:symmfoldstable}}
      \arrow[d]
      &
      \text{\ref{thm:cutperc:foldstable}}
      \arrow[d, dashed]
      \arrow[l, rounded corners,%
        to path={%
          -- ([yshift={3ex}]\tikztostart.north)%
          -- ([yshift={3ex}]\tikztotarget.north)%
          -- (\tikztotarget)}]
      \\
      \text{\ref{thm:cutperc:reachevery}}
      \arrow[r, dashed]
      \arrow[d, shift right]
      &
      \text{\ref{thm:cutperc:reach}}
      \arrow[ul]
      \arrow[ur]
      &
      \text{\ref{thm:cutperc:stronglyfoldstable}}
      \arrow[d]
      \\
      \text{\ref{thm:cutperc:llevery}}
      \arrow[u, dashed, shift right]
      \arrow[r, dashed]
      &
      \text{\ref{thm:cutperc:ll}}
      \arrow[u, dashed]
      &
      \text{\ref{thm:cutperc:maximal}}
      \arrow[ul]
    \end{tikzcd}
  \end{equation*}
  In the above, dashed arrows follow trivially from definitions and the joining arrow
  means~\ref{thm:cutperc:cutperc}$\land$\ref{thm:cutperc:foldstable}$\implies$\ref{thm:cutperc:symmfoldstable}.
  
  For the implication~\ref{thm:cutperc:cutperc}$\implies$\ref{thm:cutperc:reachevery}, since $G$ is
  $\cF$-cut-percolating, there exist a finite sequence $E_0,E_1,\ldots,E_n\subseteq E(G)$ and a
  finite sequence $(f_1,L_1),\ldots,(f_n,L_n)\in\cF$ of folds such that $\lvert E_0\rvert=1$,
  $E_n=E(G)$ and for every $i\in[n]$, $E_{i+1}=(f_i)_{L_i}^{-1}(E_i)$, that is, the composition
  \begin{equation*}
    (f_1)_{L_1}\comp\cdots (f_n)_{L_n}
  \end{equation*}
  when applied to $E_n=E(G)$ has image $E_0$, whose size is $1$, so for $c\colon E(G)\to C$, the coloring
  \begin{equation*}
    c\comp(f_1)_{L_1}\comp\cdots (f_n)_{L_n}
  \end{equation*}
  is monochromatic, hence $c\preceq_{\cF}\cM$.

  \medskip

  For the implication~\ref{thm:cutperc:reach}$\implies$\ref{thm:cutperc:cutperc}, since
  $r\preceq_{\cF}\cM$ and Lemma~\ref{lem:dualfold} implies that $\cF$ is closed under dual folds,
  there exists a sequence of folds $(f_1,L_1),\ldots,(f_n,L_n)\in\cF$ such that
  \begin{equation*}
    r\comp (f_1)_{L_1}\comp\cdots\comp(f_n)_{L_n}
  \end{equation*}
  is monochromatic. Since $r$ is injective, this means that $(f_1)_{L_1}\comp\cdots\comp(f_n)_{L_n}$
  must map $E(G)$ to a set $E_0$ of size $1$, so setting inductively
  \begin{align*}
    E_{i+1} \df (f_i)_{L_i}^{-1}(E_i)
  \end{align*}
  gives $E_n\df E(G)$, hence $G$ is $\cF$-cut-percolating.

  \medskip

  For the implication~\ref{thm:cutperc:reachevery}$\implies$\ref{thm:cutperc:llevery}, by
  Proposition~\ref{prop:abs}, it suffices to show that for every $c\colon E(G)\to C$, the folding
  problem $(G,c,\cM,\cF)$ is absorbing. For this, we need to show that if $c\preceq_{\cF} c'$,
  $(f,L)\in\cF$ and $c'\preceq_{\cF}\cM$ is such that $c'\notin\cM$, then $c'\comp
  f_L\preceq_{\cF}\cM$ and $c'\comp f_L^*\preceq_{\cF}\cM$. But since every coloring $c''\colon
  E(G)\to C$ satisfies $c''\preceq_{\cF}\cM$, this follows immediately.

  \medskip

  We now prove the implication~\ref{thm:cutperc:reach}$\implies$\ref{thm:cutperc:foldstable}. Since
  $r\preceq_{\cF}\cM$ and $\cF$ is closed under dual folds (by Lemma~\ref{lem:dualfold}), there
  exists a sequence of folds $(f_1,L_1),\ldots,(f_n,L_n)\in\cF$ such that
  \begin{equation*}
    r\comp (f_1)_{L_1}\comp\cdots\comp(f_n)_{L_n}
  \end{equation*}
  is monochromatic. Since $r$ is injective, this means that $(f_1)_{L_1}\comp\cdots\comp(f_n)_{L_n}$
  must map $E(G)$ to a set $E_0$ of size $1$, so
  \begin{equation*}
    c'\df c\comp (f_1)_{L_1}\comp\cdots\comp(f_n)_{L_n}
  \end{equation*}
  must also be monochromatic. Since $\cF$ is $\Aut(G)$-invariant and closed under dual folds and
  $c\preceq_{\cF} c'$, by Lemma~\ref{lem:foldstable}\ref{lem:foldstable:reachability}, it follows
  that $(G,c)\cong (G,c')$, so $c$ must also be monochromatic.

  \medskip

  For the implication~\ref{thm:cutperc:stronglyfoldstable}$\implies$\ref{thm:cutperc:maximal}, let
  \begin{equation*}
    \cC \df \{c'\colon E(G)\to C \mid c\preceq_{\cF} c'\}
  \end{equation*}
  be the set of colorings $\cF$-reachable from $c$. Since $c$ is $\preceq_{\cF}$-maximal, it follows
  that $\cC$ is a $\preceq_{\cF}$-strongly connected set of $\preceq_{\cF}$-maximal elements. Since
  $G$ is connected, by Lemma~\ref{lem:maximal}\ref{lem:maximal:stronglyfoldstable}, every element of
  $\cC$, in particular also $c$, is strongly $\cF$-fold-stable, hence monochromatic.

  \medskip

  For the implication~\ref{thm:cutperc:maximal}$\implies$\ref{thm:cutperc:reach}, let
  \begin{equation*}
    \cC \df \{c\colon E(G)\to C \mid r\preceq_{\cF} c\}
  \end{equation*}
  be the set of colorings $\cF$-reachable from $r$. Since $G$ is finite, we know that $\cC$ is
  finite, so it must contain at least one $\preceq_{\cF}$-maximal element $m$, which must be
  monochromatic, so $r\preceq_{\cF}\cM$.

  \medskip

  For the implication~\ref{thm:cutperc:symmfoldstable}$\implies$\ref{thm:cutperc:reach}, let
  \begin{equation*}
    \cC \df \{c\colon E(G)\to C \mid r\preceq_{\cF} c\}
  \end{equation*}
  be the set of colorings $\cF$-reachable from $r$. Since $G$ is finite, we know that $\cC$ is
  finite, so it must contain at least one $\preceq_{\cF}$-maximal element $m$. Let
  \begin{equation*}
    \cC' \df \{c\colon E(G)\to C \mid m\preceq_{\cF} c\}
  \end{equation*}
  be the set of colorings $\cF$-reachable from $m$. Since $m$ is $\preceq_{\cF}$-maximal, it follows
  that $\cC'$ is a $\preceq_{\cF}$-strongly connected set of $\preceq_{\cF}$-maximal elements. Since
  $G$ is connected, by Lemma~\ref{lem:maximal}\ref{lem:maximal:stronglyfoldstable}, every element of
  $\cC'$, in particular also $m$, is strongly $\cF$-fold-stable. On the other hand, since
  $r\preceq_{\cF} m$ by construction, Lemma~\ref{lem:foldstable}\ref{lem:foldstable:symmfoldstable}
  says $m$ is in fact symmetrically $\cF$-fold-stable, hence monochromatic. Thus
  $r\preceq_{\cF}\cM$.

  \medskip

  Finally, we prove the joining arrow, that is, we prove
  that~\ref{thm:cutperc:cutperc}$\land$\ref{thm:cutperc:foldstable}$\implies$\ref{thm:cutperc:symmfoldstable},
  it is clear that~\ref{thm:cutperc:foldstable} implies that every symmetrically $\cF$-fold-stable
  coloring is monochromatic. On the other hand, Lemma~\ref{lem:Kedgetransitive} says
  that~\ref{thm:cutperc:cutperc} implies that $G$ is $K$-edge-transitive.
\end{proof}

The following theorem is the left-sided analogue of Theorem~\ref{thm:cutperc}.

\begin{theorem}\label{thm:leftcutperc}
  Let $(G,c)$ be a finite connected bigraph, let $C$ be a set of colors, let $c\colon E(G)\to C'$ be
  a coloring of $G$, let $\cF\subseteq\Fold((G,c))$ be a set of folds of $(G,c)$ that is
  $\Aut((G,c))$-invariant, let
  \begin{equation*}
    \cM \df \{\ell\otimes c \mid \ell\colon V_1(G)\to C\text{ is monochromatic}\}
  \end{equation*}
  and let $r\colon V_1(G)\to C$ be a rainbow left-coloring of $G$. Then the following are
  equivalent:
  \begin{enumerate}
  \item\label{thm:leftcutperc:leftcutperc} $G$ is $\cF$-left-cut-percolating.
  \item\label{thm:leftcutperc:reachevery} $\ell\otimes c\preceq_{\cF}\cM$ for every left-coloring
    $\ell\colon V_1(G)\to C$.
  \item\label{thm:leftcutperc:reach} $r\otimes c\preceq_{\cF}\cM$.
  \item\label{thm:leftcutperc:llevery} $\ell\otimes c\ll_{\cF}\cM$ for every left-coloring
    $\ell\colon V_1(G)\to C$.
  \item\label{thm:leftcutperc:ll} $r\otimes c\ll_{\cF}\cM$.
  \item\label{thm:leftcutperc:foldstable} If $\ell\colon V_1(G)\to C$ is a left-coloring of $G$ such
    that $\ell\otimes c$ is $(\Aut((G,c)),\cF)$-fold-stable, then $\ell$ is monochromatic.
  \item\label{thm:leftcutperc:stronglyfoldstable} If $\ell\colon V_1(G)\to C$ is a left-coloring of
    $G$ such that $\ell\otimes c$ is strongly $(\Aut((G,c)),\cF)$-fold-stable, then $\ell$ is
    monochromatic.
  \item\label{thm:leftcutperc:maximal} If $\ell\colon V_1(G)\to C$ is a left-coloring of $G$ such that
    $\ell\otimes c$ is $\preceq_{\cF}$-maximal, then $\ell$ is monochromatic.
  \item\label{thm:leftcutperc:leftsymmfoldstable} $G$ is $K$-left-vertex-transitive, where $K$ is
    the subgroup of $\Aut((G,c))$ generated by $\{f\mid (f,L)\in\cF\}$ and for every left-coloring
    $\ell\colon V_1(G)\to C$ of $G$, if $\ell\otimes c$ left-symmetrically
    $(\Aut((G,c)),\cF)$-fold-stable, then $\ell$ is monochromatic.
  \end{enumerate}
\end{theorem}

\begin{proof}
  The proof is analogous to that of Theorem~\ref{thm:cutperc}, except for the following:
  \begin{enumerate}[label={\arabic*.}]
  \item We can trivially upgrade $\cF$-fold-stability of a coloring of the form $\ell\otimes c$ to
    $(\Aut((G,c)),\cF)$-fold-stability. This is because for every fold $(f,L)\in\cF$, we have
    $(\ell\otimes c)\comp f_L = (\ell\comp f_L)\otimes c$, so any isomorphism between
    $(G,\ell\otimes c)$ and $(G,(\ell\otimes c)\comp f_L)$ must necessarily preserve $c$.
  \item For left-symmetric fold-stability, we will use Lemma~\ref{lem:leftsymmfoldstable} instead of
    Lemma~\ref{lem:foldstable}\ref{lem:foldstable:symmfoldstable}.
  \end{enumerate}

  The structure of the proof is as follows:
  
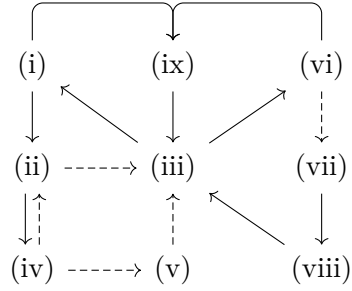
\begin{equation*}
    \begin{tikzcd}
      \text{\ref{thm:leftcutperc:leftcutperc}}
      \arrow[d]
      \arrow[r, rounded corners,%
        to path={%
          -- ([yshift={3ex}]\tikztostart.north)%
          -- ([yshift={3ex}]\tikztotarget.north)%
          -- (\tikztotarget)}]
      &
      \text{\ref{thm:leftcutperc:leftsymmfoldstable}}
      \arrow[d]
      &
      \text{\ref{thm:leftcutperc:foldstable}}
      \arrow[d, dashed]
      \arrow[l, rounded corners,%
        to path={%
          -- ([yshift={3ex}]\tikztostart.north)%
          -- ([yshift={3ex}]\tikztotarget.north)%
          -- (\tikztotarget)}]
      \\
      \text{\ref{thm:leftcutperc:reachevery}}
      \arrow[r, dashed]
      \arrow[d, shift right]
      &
      \text{\ref{thm:leftcutperc:reach}}
      \arrow[ul]
      \arrow[ur]
      &
      \text{\ref{thm:leftcutperc:stronglyfoldstable}}
      \arrow[d]
      \\
      \text{\ref{thm:leftcutperc:llevery}}
      \arrow[u, dashed, shift right]
      \arrow[r, dashed]
      &
      \text{\ref{thm:leftcutperc:ll}}
      \arrow[u, dashed]
      &
      \text{\ref{thm:leftcutperc:maximal}}
      \arrow[ul]
    \end{tikzcd}
  \end{equation*}
  In the above, dashed arrows follow trivially from definitions and the joining arrow
  means~\ref{thm:leftcutperc:leftcutperc}$\land$\ref{thm:leftcutperc:foldstable}$\implies$\ref{thm:leftcutperc:leftsymmfoldstable}.

  For the implication~\ref{thm:leftcutperc:leftcutperc}$\implies$\ref{thm:leftcutperc:reach}, since
  $G$ is $\cF$-left-cut-percolating, there exist a finite sequence $V_0,V_1,\ldots,V_n\subseteq
  V_1(G)$ and a finite sequence $(f_1,L_1),\ldots,(f_n,L_n)\in\cF$ of folds such that $\lvert
  V_0\rvert=1$, $V_n=V_1(G)$ and for every $i\in[n]$, $V_{i+1}=(f_i)_{L_i}^{-1}(V_i)$, that is, the
  composition
  \begin{equation*}
    (f_1)_{L_1}\comp\cdots (f_n)_{L_n}
  \end{equation*}
  when applied to $V_n=V_1(G)$ has image $V_0$, whose size is $1$, so
  \begin{equation*}
    \ell\comp(f_1)_{L_1}\comp\cdots (f_n)_{L_n}
  \end{equation*}
  is monochromatic, hence $\ell\otimes c\preceq_{\cF}\cM$.

  \medskip

  For the implication~\ref{thm:leftcutperc:reach}$\implies$\ref{thm:leftcutperc:leftcutperc}, since
  $r\otimes c\preceq_{\cF}\cM$ and Lemma~\ref{lem:dualfold} implies that $\cF$ is closed under dual
  folds, there exists a sequence of folds $(f_1,L_1),\ldots,(f_n,L_n)\in\cF$ such that
  \begin{equation*}
    r\comp (f_1)_{L_1}\comp\cdots\comp(f_n)_{L_n}
  \end{equation*}
  is monochromatic. Since $r$ is injective, this means that $(f_1)_{L_1}\comp\cdots\comp(f_n)_{L_n}$
  must map $V_1(G)$ to a set $V_0$ of size $1$, so setting inductively
  \begin{align*}
    V_{i+1} \df (f_i)_{L_i}^{-1}(V_i)
  \end{align*}
  gives $V_n\df V_1(G)$, hence $G$ is $\cF$-cut-percolating.

  \medskip

  For the implication~\ref{thm:leftcutperc:reachevery}$\implies$\ref{thm:leftcutperc:llevery}, by
  Proposition~\ref{prop:abs}, it suffices to show that for every $\ell\colon V_1(G)\to C$, the
  folding problem $(G,\ell\otimes c,\cM,\cF)$ is absorbing. For this, we need to show that if
  $\ell\otimes c\preceq_{\cF} \ell'\otimes c$, $(f,L)\in\cF$ and $\ell'\otimes c\preceq_{\cF}\cM$ is
  such that $\ell'\otimes c\notin\cM$, then
  \begin{align*}
    (\ell'\otimes c)\comp f_L & = (\ell'\comp f_L)\otimes c\preceq_{\cF}\cM, &
    (\ell'\otimes c)\comp f_L^* & = (\ell'\comp f_L^*)\otimes c\preceq_{\cF}\cM,
  \end{align*}
  But since every left-coloring $\ell''\colon E(G)\to C$ satisfies $\ell''\otimes
  c\preceq_{\cF}\cM$, this follows immediately.

  \medskip

  We now prove the
  implication~\ref{thm:leftcutperc:reach}$\implies$\ref{thm:leftcutperc:foldstable}. Since $r\otimes
  c\preceq_{\cF}\cM$ and $\cF$ is closed under dual folds (by Lemma~\ref{lem:dualfold}), there
  exists a sequence of folds $(f_1,L_1),\ldots,(f_n,L_n)\in\cF$ such that
  \begin{equation*}
    r\comp (f_1)_{L_1}\comp\cdots\comp(f_n)_{L_n}
  \end{equation*}
  is monochromatic. Since $r$ is injective, this means that $(f_1)_{L_1}\comp\cdots\comp(f_n)_{L_n}$
  must map $V_1(G)$ to a set $V_0$ of size $1$, so
  \begin{equation*}
    \ell'\df \ell\comp (f_1)_{L_1}\comp\cdots\comp(f_n)_{L_n}
  \end{equation*}
  must also be monochromatic. Since $\cF$ is $\Aut((G,c))$-invariant and closed under dual folds and
  $\ell\otimes c\preceq_{\cF} \ell'\otimes c$, by
  Lemma~\ref{lem:foldstable}\ref{lem:foldstable:reachability}, it follows that $(G,\ell\otimes
  c)\cong (G,\ell'\otimes c)$, so $\ell$ must also be monochromatic.

  \medskip

  For the
  implication~\ref{thm:leftcutperc:stronglyfoldstable}$\implies$\ref{thm:leftcutperc:maximal}, let
  \begin{equation*}
    \cC \df \{\ell'\colon V_1(G)\to C \mid \ell\otimes c\preceq_{\cF} \ell'\otimes c\}
  \end{equation*}
  be the set of left-colorings $\ell'$ such that $\ell'\otimes c$ is $\cF$-reachable from
  $\ell\otimes c$. Since $\ell\otimes c$ is $\preceq_{\cF}$-maximal, it follows that
  \begin{equation*}
    \cC\otimes\{c\} \df \{\ell'\otimes c \mid \ell'\in\cC\}
  \end{equation*}
  is a $\preceq_{\cF}$-strongly connected set of $\preceq_{\cF}$-maximal elements. Since $G$ is
  connected, by Lemma~\ref{lem:maximal}\ref{lem:maximal:stronglyfoldstable}, every element of
  $\cC\otimes\{c\}$, in particular also $\ell\otimes c$, is strongly
  $(\Aut((G,c)),\cF)$-fold-stable, hence monochromatic.

  \medskip

  For the implication~\ref{thm:leftcutperc:maximal}$\implies$\ref{thm:leftcutperc:reach}, let
  \begin{equation*}
    \cC \df \{\ell\colon V_1(G)\to C \mid r\otimes c\preceq_{\cF} \ell\otimes c\}
  \end{equation*}
  be the set of left-colorings $\ell$ such that $\ell\otimes c$ is $\cF$-reachable from $r\otimes
  c$. Since $G$ is finite, we know that $\cC$ is finite, so there must exist at least one $m\in\cC$
  such that $m\otimes c$ is $\preceq_{\cF}$-maximal. By assumption, $m$ must be monochromatic, so
  $r\otimes c\preceq_{\cF}\cM$.

  \medskip

  For the implication~\ref{thm:leftcutperc:leftsymmfoldstable}$\implies$\ref{thm:leftcutperc:reach}, let
  \begin{equation*}
    \cC \df \{\ell\colon V_1(G)\to C \mid r\otimes c\preceq_{\cF} \ell\otimes c\}
  \end{equation*}
  be the set of left-colorings $\ell$ such that $\ell\otimes c$ is $\cF$-reachable from $r\otimes
  c$. Since $G$ is finite, we know that $\cC$ is finite, so there must exist at least one $m\in\cC$
  such that $m\otimes c$ is $\preceq_{\cF}$-maximal. Let
  \begin{equation*}
    \cC' \df \{\ell\colon V_1(G)\to C \mid m\otimes c\preceq_{\cF} \ell\otimes c\}
  \end{equation*}
  be the set of left-colorings $\ell$ such that $\ell\otimes c$ is $\cF$-reachable from $m\otimes
  c$. Since $m\otimes c$ is $\preceq_{\cF}$-maximal, it follows that
  \begin{equation*}
    \cC'\otimes\{c\} \df \{\ell\otimes c \mid \ell\in\cC'\}
  \end{equation*}
  is a $\preceq_{\cF}$-strongly connected set of $\preceq_{\cF}$-maximal elements. Since $G$ is
  connected, by Lemma~\ref{lem:maximal}\ref{lem:maximal:stronglyfoldstable}, every element of
  $\cC'$, in particular also $m$, is strongly $(\Aut((G,c)),\cF)$-fold-stable. On the other hand,
  since $r\otimes c\preceq_{\cF} m\otimes c$ by construction, Lemma~\ref{lem:leftsymmfoldstable}
  says $m\otimes c$ is in fact symmetrically $(\Aut((G,c)),\cF)$-fold-stable, hence
  monochromatic. Thus $r\otimes c\preceq_{\cF}\cM$.

  \medskip

  Finally, we prove the joining arrow, that is, we prove
  that~\ref{thm:leftcutperc:leftcutperc}$\land$\ref{thm:leftcutperc:foldstable}$\implies$\ref{thm:leftcutperc:leftsymmfoldstable},
  it is clear that~\ref{thm:leftcutperc:foldstable} implies that every symmetrically
  $(\Aut((G,c)),\cF)$-fold-stable coloring is monochromatic. On the other hand,
  Lemma~\ref{lem:Kleftvertextransitive} says that~\ref{thm:leftcutperc:leftcutperc} implies that $G$
  is $K$-left-vertex-transitive.
\end{proof}

\appendix

\section{Upgrading connected core flag isomorphism}
\label{sec:coreiso}

\begin{proposition}
  If $F_1$ and $F_2$ are flags such that $F_1\rest_{C(F_1)}\cong F_2\rest_{C(F_2)}$ and $H(F_1)\cong
  H(F_2)$, then $F_1\cong F_2$.
\end{proposition}

\begin{proof}
  Let $f\colon C(F_1)\to C(F_2)$ be an isomorphism from $F_1\rest_{C(F_1)}$ to $F_2\rest_{C(F_2)}$, let
  $g\colon V(F_1)\to V(F_2)$ be an isomorphism from $H(F_1)$ to $H(F_2)$ and let $\cC_i$ be the set of
  connected components of $G(F_i)$ ($i\in[2]$).

  By the definition of connected core, we know that each $C\in\cC_i$ is either contained in $C(F_i)$
  or is disjoint from it. Let $\cC'_i\df\{C\in\cC_i \mid C\subseteq C(F_i)\}$ and note that since
  $g$ and $f$ are isomorphisms, it follows that $g$ induces a bijection from $\cC_1$ to $\cC_2$ and
  $f$ induces a bijection from $\cC'_1$ to $\cC'_2$.

  Note also that $\cC'_1$ and $\cC'_2$ are finite (in fact, they have cardinality at most
  $k_{F_1}=k_{F_2}$) as every element of $\cC'_i$ must contain at least one point in the image of
  $\theta_{F_i}$.

  Let
  \begin{equation*}
    \cD \df \cC_2\setminus g(\cC'_1)
  \end{equation*}
  and for each $D\in\cD$, let us define $t_D$ as the minimum $t\in\NN$ such that
  \begin{equation*}
    (g\comp f^{-1})^t(D)
    \df
    (\mathop{\underbrace{(g\comp f^{-1})\comp\cdots\comp (g\comp f^{-1})}}\limits_{t\text{ times}})(D)
    \notin\cC'_2
  \end{equation*}
  (Note that the expression inductively makes sense: if for $t-1$ the above is in $\cC'_2$, then we
  can apply $g\comp f^{-1}$ to obtain the expression for $t$.) Let us argue that this $t_D$ indeed
  exists and in fact $t_D\leq k_{F_2}$. Suppose not, then for each $t\in\{0,\ldots,k_{F_2}\}$, the
  connected component
  \begin{equation*}
    C_t\df (g\comp f^{-1})^t(D)
  \end{equation*}
  is in $\cC'_2$. Since $\lvert\cC'_2\rvert\leq k_{F_2}$, there must be repetitions among the $C_t$
  and since $g\comp f^{-1}$ is injective on the finite set $\cC'_2$, it must also be bijective on
  this set, which in particular implies that it is bijective on $\{C_t \mid
  t\in\{0,\ldots,k_{F_2}\}$, hence there must exist $t_*\in\{0,\ldots,k_{F_2}\}$ such that $(g\comp
  f^{-1})(C_{t_*}) = C_0 = D$. Since $f^{-1}(C_{t_*})\in\cC'_1$, this contradicts $D\notin
  g(\cC'_1)$.

  We now define $h\colon V(F_1)\to V(F_2)$ by
  \begin{equation*}
    h(v) \df
    \begin{dcases*}
      f(v), & if $v\in C(F_1)$,\\
      ((g\comp f^{-1})^{t_D}\comp g)(v), & if $v\in g^{-1}(D)$ for $D\in\cD$.
    \end{dcases*}
  \end{equation*}
  (Note that there is no ambiguity in the definition above as $g^{-1}(\cD)=\cC_1\setminus\cC'_1$ and
  each $v$ is in only one connected component.)

  It is clear that if $v$ and $w$ are in the same connected component of $F_1$, then they land on
  the same case in the definition of $h$. In particular, this implies that $h$ preserves edges,
  non-edges and the coloring. Since $h$ acts as $f$ on $C(F_1)$, it is also clear that $h$ preserves
  the flag labeling.

  It remains only to show that $h$ is indeed bijective. Since both $f$ and $g$ are bijective, it
  suffices to show that
  \begin{equation*}
    (f(C(F_1)), ((g\comp f^{-1})^{t_D}(D))_{D\in\cD})
  \end{equation*}
  forms a partition of $V(F_2)$. By inspecting connected components, it is clear that the above are
  pairwise disjoint, so it remains to show that their union is $V(F_2)$. Inspecting connected
  components again, it suffices to show that
  \begin{equation*}
    \cC'_2\cup\bigcup_{D\in\cD} (g\comp f^{-1})^{t_D}(D) = \cC_2,
  \end{equation*}
  which in turn reduces to showing that if $C\in\cC_2\setminus\cC'_2$, then there exists $D\in\cD$
  with $(g\comp f^{-1})^{t_D}(D)=C$.

  Fix one such $C\in\cC_2\setminus\cC'_2$ and let $u_C\in\NN$ be the minimum $u$ such that
  \begin{equation*}
    (f\comp g^{-1})^u(C)\in\cD
  \end{equation*}
  (Note that the expression above inductively makes sense: if for $u-1$ the above is not in $\cD$,
  then it must be in $g(\cC'_1)$, so we can apply $f\comp g^{-1}$ to it to obtain the expression for
  $u$.) A similar argument to that of $t_D$ shows that $u_C$ indeed exists and $u_C\leq k_{F_1}$.

  Let now $D\df (f\comp g^{-1})^{u_C}(C)\in\cD$ and let us show that $(g\comp
  f^{-1})^{t_D}(D)=C$. To do so, it suffices to show that $u_C=t_D$. Since $(g\comp
  f^{-1})^{u_C}(D)=C\notin\cC'_2$, it is clear that $t_D\leq u_C$.

  To see the other inequality, note that the definition of $t_D$ implies that
  \begin{equation*}
    \cC_2\setminus\cC'_2
    \ni
    (g\comp f^{-1})^{t_D}(D)
    =
    (f\comp g^{-1})^{u_C-t_D}(C)
  \end{equation*}
  and since the image of $f$ is contained in $\cC'_2$, we must have $u_c-t_D=0$.

  Thus $h$ is an isomorphism from $F_1$ to $F_2$.
\end{proof}

\bibliographystyle{alpha}
\bibliography{refs}

\begin{thebibliography}{KlMPW19}

\bibitem[CL17]{CL17}
David Conlon and Joonkyung Lee.
\newblock Finite reflection groups and graph norms.
\newblock {\em Adv. Math.}, 315:130--165, 2017.

\bibitem[CL24]{CL24}
David Conlon and Joonkyung Lee.
\newblock Domination inequalities and dominating graphs, 2024.

\bibitem[Cor24]{Cor24}
Leonardo~N. Coregliano.
\newblock Left-cut-percolation and induced-{S}idorenko bigraphs.
\newblock {\em SIAM J. Discrete Math.}, 38(2):1586--1629, 2024.

\bibitem[CR21]{CR21}
Leonardo~N. Coregliano and Alexander~A. Razborov.
\newblock Biregularity in sidorenko's conjecture, 2021.

\bibitem[Hat10]{Hat10}
Hamed Hatami.
\newblock Graph norms and {S}idorenko's conjecture.
\newblock {\em Israel J. Math.}, 175:125--150, 2010.

\bibitem[KlMPW19]{KMPW19}
Daniel Kr\'a\v~l, Ta\'isa~L. Martins, P\'eter~P\'al Pach, and Marcin Wrochna.
\newblock The step {S}idorenko property and non-norming edge-transitive graphs.
\newblock {\em J. Combin. Theory Ser. A}, 162:34--54, 2019.

\bibitem[Lov12]{Lov12}
L\'aszl\'o Lov\'asz.
\newblock {\em Large networks and graph limits}, volume~60 of {\em American
  Mathematical Society Colloquium Publications}.
\newblock American Mathematical Society, Providence, RI, 2012.

\bibitem[LS21]{LS21}
Joonkyung Lee and Bjarne Sch\"ulke.
\newblock Convex graphon parameters and graph norms.
\newblock {\em Israel J. Math.}, 242(2):549--563, 2021.

\bibitem[LS22]{LS22}
Joonkyung Lee and Alexander Sidorenko.
\newblock On graph norms for complex-valued functions.
\newblock {\em J. Lond. Math. Soc. (2)}, 106(2):1501--1538, 2022.

\bibitem[Raz07]{Raz07}
Alexander~A. Razborov.
\newblock Flag algebras.
\newblock {\em J. Symbolic Logic}, 72(4):1239--1282, 2007.

\bibitem[Sid20]{Sid20}
Alexander Sidorenko.
\newblock Weakly norming graphs are edge-transitive.
\newblock {\em Combinatorica}, 40(4):601--604, 2020.

\end{thebibliography}

\end{document}